%% file: rcd_paper.tex
\documentclass[]{article}
\usepackage{proceed2e}

\usepackage{times}

\usepackage[utf8]{inputenc}
\include{header}

\usepackage{amsmath}
\usepackage{amsfonts}
\usepackage{amsthm}
\usepackage{bm}
\usepackage{algorithmic}
\usepackage{wrapfig}

\usepackage{xr}

\newtheorem{assumption}{Assumption}

\newcommand{\nb}[1]{\bar{#1}}
\DeclareMathOperator*{\argmin}{arg\,min}

\newcolumntype{x}[1]{>{\centering\arraybackslash\hspace{0pt}}p{#1}}

\title{Large-scale randomized-coordinate descent methods with non-separable linear constraints}

\newcommand*\samethanks[1][\value{footnote}]{\footnotemark[#1]}

\author{
{\bf Sashank J. Reddi \thanks{\ \ Indicates equal contribution.}}
\quad 
{\bf Ahmed Hefny\samethanks[1]}
\quad 
{\bf Carlton Downey}
\quad
{\bf Avinava Dubey}
\\
Machine Learning Department \\
Carnegie Mellon University
\\ \quad \\ 
{\bf Suvrit Sra \thanks{\ \ A part of this work was performed when the author was at Carnegie Mellon University}} \\
Massachusetts Institute of Technology
}

\begin{document}

\maketitle

\begin{abstract}
  We develop randomized block coordinate descent (CD) methods for linearly constrained convex optimization. Unlike other large-scale CD methods, we do not assume the constraints to be separable, but allow them be coupled linearly. To our knowledge, ours is the first CD method that allows linear coupling constraints, without making the global iteration complexity have an \emph{exponential dependence} on the number of constraints. We present algorithms and theoretical analysis for four key (convex) scenarios: (i) smooth; (ii) smooth + separable nonsmooth; (iii) asynchronous parallel; and (iv) stochastic. We discuss some architectural details of our methods and present  preliminary results to illustrate the behavior of our algorithms.
\end{abstract}

\section{INTRODUCTION}
Coordinate descent (CD) methods are conceptually among the simplest schemes for unconstrained optimization---they have been studied for a long time (see e.g.,~\citep{auslender,polyak,bertsekas99}), and are now enjoying greatly renewed interest. Their resurgence is rooted in successful applications in machine learning~\citep{hsiehCha08,hsiehDhil11}, statistics~\citep{fried07,hsiehSice}, and many other areas---see~\citep{shwZha13,richTak11,richtarik:13} and references therein for more examples.

A catalyst to the theoretical as well as practical success of CD methods has been randomization. (The idea of randomized algorithms for optimization methods is of course much older, see e.g.,~\citep{rastrigin}.) Indeed, generic non-randomized CD has resisted complexity analysis, though there is promising recent work~\citep{saha,wangLin14,hongWang}; remarkably for randomized CD for smooth convex optimization, Nesterov~\citep{Nesterov10,nesterov:12} presented an analysis of global iteration complexity. This work triggered several improvements, such as \citep{richTak11,richTak12}, who simplified and extended the analysis to include separable nonsmooth terms. Randomization has also been crucial to a host of other CD algorithms and analyses~\citep{hsiehCha08,bradKyr11,necoara:11,shwZha13,shwZha13b,tappen13,richTak12,Recht11,Liu13,richtarik:13}. 

Almost all of the aforementioned CD methods assume essentially unconstrained problems, which at best allow separable constraints. In contrast, we develop, analyze, and implement randomized CD methods for the following composite objective convex problem with \emph{non-separable} linear constraints
\begin{align}
\label{eq:opt}
\nlmin_x F(x) := f(x) + h(x)\quad\text{ s.t. }\quad Ax=0.
\end{align}   
Here, $f: \reals^n \rightarrow \reals$ is assumed to be continuously differentiable and convex, while $h: \reals^n \rightarrow \reals \cup \{\infty\}$ is lower semi-continuous, convex, coordinate-wise separable, but not necessarily smooth; the linear constraints (LC) are specified by a matrix $A \in \reals^{m \times n}$, for which $m \ll n$, and a certain structure (see \S\ref{sec:conv}) is assumed. The reader may wonder whether one cannot simply rewrite Problem~\eqref{eq:opt} in the form of $f + h$ (without additional constraints) using suitable indicator functions. However, the resulting regularized problem then \emph{no longer} fits the known efficient CD frameworks~\citep{richTak11}, since the nonsmooth part is \emph{not} block-separable.

Problem~\eqref{eq:opt} subsumes the usual regularized optimization problems pervasive in machine learning for the simplest ($m=0$) case. In the presence of  linear constraints ($m>0$), Problem~\eqref{eq:opt} assumes a form used in the classic Alternating Direction Method of Multipliers (ADMM)~\citep{gabay67,glowin}. The principal difference between our approach and ADMM is that the latter treats the entire variable $x \in \reals^n$ as a single block, whereas we use the structure of $A$ to split $x$ into $b$ smaller blocks. Familiar special cases of Problem~\eqref{eq:opt} include SVM (with bias) dual, fused Lasso and group Lasso~\citep{tibshirani2005sparsity}, and linearly constrained least-squares regression~\citep{laha74,golub96}. 

Recently, \citet{necoara:11} studied a special case of Problem~\eqref{eq:opt} that sets $h\equiv 0$ and assumes a single sum constraint. They presented a randomized CD method that starts with a feasible solution and at each iteration updates a pair of coordinates to ensure descent on the objective while maintaining feasibility. This scheme is reminiscent of the well-known SMO procedure for SVM optimization~\citep{Platt99}. For smooth convex problems with $n$ variables, \citet{necoara:11} prove an $O(1/\epsilon)$ rate of convergence. More recently, in~\citep{Necoara14} considered a generalization to the general case $Ax=0$ (assuming $h$ is coordinatewise separable). 

Unfortunately, the analysis in~\citep{Necoara14} yields an extremely pessimistic complexity result:\vspace*{-5pt}
\begin{theorem}[\protect{\citep{Necoara14}}]
  Consider Problem~\eqref{eq:opt} with $h$ being coordinatewise separable, and $A \in \reals^{m \times n}$ with $b$ blocks. Then, the CD algorithm in~\citep{Necoara14} takes no more than 
  $O(b^m/\epsilon)$ iterations to obtain a solution of $\epsilon$-accuracy.
\vspace*{-8pt}
\end{theorem}
This result is exponential in the number of constraints and too severe even for small-scale problems! 

We present randomized CD methods, and prove that for important special cases (mainly $h \equiv 0$
\emph{or} $A$ is a sum constraint)
we can obtain global iteration complexity that \emph{does not have} an intractable dependence on either the number of coordinate blocks ($b$), or on the number of linear constraints ($m$). Previously, \citet{Tseng09} also studied a linearly coupled block-CD method based on the Gauss-Southwell choice; however, their complexity analysis applies only to the special $m=0$ and $m=1$ cases. 

To our knowledge, ours is the first work on CD for problems with more than one ($m>1$) linear constraints 
that presents such results.

\paragraph{Contributions.} In light of the above background, the primary contributions of this paper are as follows:
\begin{list}{$\circ$}{\leftmargin=2em}
  \setlength{\itemsep}{-1pt}
\item Convergence rate analysis of a randomized block-CD method for the smooth case ($h\equiv 0$) with $m \ge 1$ general linear constraints.
\item A tighter convergence analysis for the composite function optimization ($h \neq 0$) than \cite{Necoara14} in the case of sum constraint.
\item An asynchronous CD algorithm for Problem~\eqref{eq:opt}.
\item A stochastic CD method with convergence analysis for solving problems with a separable loss $f(x) = (1/N) \sum_{i=1}^N f_i(x)$.
\end{list}

Table~\ref{tab:algos} summarizes our contributions and compares it with  existing state-of-the-art coordinate descent methods. The detailed proofs of all our theoretical claims are available in the appendix.

\begin{table}[tbp]
  \centering
    \begin{tabular}{ | x{1cm} || x{1cm} | x{1cm} | x{1.25cm} | x{1.7cm} |}
    \hline
     Paper & LC & Prox & Parallel & Stochastic \\ \hline
    \cite{necoara:11} & YES & $\times$ & $\times$ & $\times$ \\ \hline
    \cite{Tseng09}     & YES & YES & $\times$ & $\times$ \\ \hline
    \citep{Necoara14} & YES & YES & $\times$ & $\times$ \\ \hline
    \cite{fercoq:13}  & $\times$ & YES & YES & $\times$  \\ \hline 
    \cite{bradKyr11}  & $\times$ & $\ell_1$ & YES & YES  \\ \hline 
    \cite{richTak12}  & $\times$ & YES & YES & $\times$  \\ \hline 
    Ours & {\bf YES} & {\bf YES} & {\bf YES} & {\bf YES}  \\ \hline 
    \end{tabular}
    \caption{\small Summary comparison of our method with other CD methods; LC denotes `linear constraints'; Prox signifies an extension using proximal operators (to handle $h \neq 0$).}
    \label{tab:algos}
\end{table}

\textbf{Additional related work.} As noted, CD methods have a long history in optimization and they have gained tremendous recent interest. We cannot hope to do full justice to all the related work, but refer the reader to \citep{richTak11,richTak12} and \citep{Liu13} for more thorough coverage. Classically, local linear convergence was analyzed in~\citep{luoTseng92}. Global rates for randomized block coordinate descent (BCD) were pioneered by~\citet{Nesterov10}, and have since then been extended by various authors~\citep{richTak12,richTak11,wangLin14,beckLuba13}. The related family of Gauss-Seidel like analyses for ADMM have also recently gained prominence~\citep{hong12}. A combination of randomized block-coordinate ideas with Frank-Wolfe methods was recently presented in~\citep{lacoste12}, though algorithmically the Frank-Wolfe approach is very different as it relies on non projection based oracles.

\vspace*{-6pt}
\section{PRELIMINARIES}
\label{sec:prelims}
\vspace*{-6pt}
In this section, we further explain our model and assumptions. We assume that the entire space $\reals^n$ is decomposed into $b$ \emph{blocks}, i.e., $x = [x_1^{\top}, \cdots, x_b^{\top}]^{\top}$ where $x \in \reals^n$, $x_i \in \reals^{n_i}$ for all $i \in [b]$, and $n = \sum_i n_i$. For any $x \in \reals^n$, we use $x_i$ to denote the $i^{\text{th}}$ block of $x$. We model communication constraints in our algorithms by viewing variables as nodes in a connected graph $G := (V, E)$. Specifically, node $i \in V \equiv [b]$ corresponds to variable $x_i$, while an edge $(i,j) \in E \subset V \times V$ is present if nodes $i$ and $j$ can exchange information. We use ``pair'' and ``edge'' interchangeably.


For a differentiable function $f$, we use $f_{i_1\cdots i_p}$ and $\nabla_{i_1\cdots i_p} f(x)$ (or $\nabla_{x_{i_1}\cdots x_{i_p}} f(x)$) to denote the restriction of the function and its partial gradient to coordinate blocks $(x_{i_1}, \cdots, x_{i_p})$. For any matrix $B$ with $n$ columns, 
we use $B_i$ to denote the columns of $B$ corresponding to $x_i$ 
and $B_{ij}$ to denote the columns of $B$ corresponding to $x_i$ and $x_j$. 
We use $U$ to denote the $n \times n$ identity matrix
and hence $U_i$ is a matrix that places an $n_i$ dimensional vector 
into the corresponding block of an $n$ dimensional vector.

We make the following standard assumption on the partial gradients of $f$.

\begin{assumption}
\label{ass:lipschitz-gradient}
The function has block-coordinate Lipschitz continuous gradient, i.e.,
$$
\| \nabla_i f(x) - \nabla_i f(x + U_i h) \| \leq L_i \|h_i\| \text{ for all $x \in \mathbb{R}^n$, }.
$$
\end{assumption}

Assumption \ref{ass:lipschitz-gradient} is similar to
the typical Lipschitz continuous gradients assumed in first-order methods
and it is necessary to ensure convergence of block-coordinate methods.
When functions $f_i$ and $f_j$ have Lipschitz continuous gradients with constants $L_i$ and $L_j$ respectively, one can show that the function $f_{ij}$ has a Lipschitz continuous gradient with $L_{ij} = L_i + L_j$~\citep[Lemma~1]{Necoara14}. The following result is standard.

\begin{lemma}
For any function $g: \mathbb{R}^n \rightarrow \mathbb{R}$ with $L$-Lipschitz continuous gradient $\nabla g$, we have
$$
g(x) \leq g(y) + \langle \nabla g(y), x - y \rangle + \tfrac{L}{2} \|x - y\|^2 \ x,y \in \mathbb{R}^n.
$$
\label{lem:descent-lemma}\vspace*{-15pt}
\end{lemma}
Following~\citep{richTak11,Tseng09}, we also make the following assumption on the structure of $h$. 
\begin{assumption}
\label{ass:coordinate-separable}
The nonsmooth function $h$ is block separable, i.e., $h(x) = \sum_i h_i(x_i)$.
\end{assumption}
This assumption is critical to composite optimization using CD methods. We also assume access to an oracle that returns function values and \emph{partial gradients} at any points and iterates of the optimization algorithm.


\vspace*{-8pt}
\section{ALGORITHM}
\label{sec:algo}
\vspace*{-8pt}
We are now ready to present our randomized CD methods for Problem~\eqref{eq:opt} in various settings. We first study composite minimization (\S\ref{sec:comp}) and later look at asynchronous (\S\ref{sec:async}) and stochastic (\S\ref{sec:stoch}) variants. The main idea underlying our algorithms is to pick a random pair $(i,j) \in E$ of variables (blocks) at each iteration, and to update them in a manner which maintains feasibility and ensures progress in optimization.

\vspace*{-5pt}
\subsection{Composite Minimization}
\label{sec:comp}
\vspace*{-5pt}
We begin with the nonsmooth setting, where $h\not\equiv 0$. We start with a feasible point $x^0$. Then, at each iteration we pick a random pair $(i,j) \in E$ of variables and minimize the first-order Taylor expansion of the loss $f$ around the current iterate while maintaining feasibility. 
Formally, this involves performing the update
\begin{align}
\label{eq:subproblem}
Z(f,x,(i,j),\alpha) := \argmin_{A_{ij}d_{ij} = 0} f(x) + \langle \nabla_{ij} f(x), d_{ij} \rangle \\
\nonumber + {(2\alpha)}^{-1} \|d_{ij}\|^2 + h(x + U_{ij} d_{ij}),
\end{align}
where $\alpha > 0$ is a stepsize parameter and $d_{ij}$ is the update. The right hand side of Equation~\eqref{eq:subproblem} upper bounds $f$ at $x+U_{ij}d_{ij}$, as seen by using Assumption~\ref{ass:lipschitz-gradient} and Lemma~\ref{lem:descent-lemma}. 
If $h(x) \equiv 0$, minimizing Equation~\eqref{eq:subproblem} yields
\begin{align}
\nonumber \lambda \gets \alpha(A_i A_i^\top & + A_j A_j^\top)^+\left(A_i \nabla_i f(x) + A_j \nabla_j f(x)\right) \nonumber \\
\nonumber d_i & \gets - \alpha \nabla_i f(x) + A_i^\top \lambda  \\
d_j & \gets - \alpha \nabla_j f(x) + A_j^\top \lambda 
\label{eq:primal_update}
\end{align}
Algorithm~\ref{alg:composite-minimization} presents the resulting method.

Note that since we start with a feasible point $x^0$ and the update $d^k$ satisfies $A d^k = 0$, the iterate $x^k$ is always feasible.
However, it can be shown that a necessary condition for Equation~$\eqref{eq:subproblem}$ to result in a non-zero update
is that $A_i$ and $A_j$ span the same column space.
If the constraints are not block separable (i.e. for any partitioning of blocks $x_1,\dots,x_b$
into two groups, there is a constraint that involves blocks from both groups),
a typical way to satisfy the aforementioned condition is to require $A_i$ to be full row-rank for all $i \in [b]$.
This constraints the minimum block size to be chosen in order to apply randomized CD.

Theorem~\ref{thm:conv} describes convergence of Algorithm~\ref{alg:composite-minimization} for the smooth case ($h\equiv 0$), while Theorem~\ref{thm:non-smooth-conv} considers the nonsmooth case under a suitable assumption on the structure of the interdependency graph $G$---both results are presented in Section~\ref{sec:conv}.

\begin{algorithm}[htbp]
  \caption{{\small Composite Minimization with Linear Constraints}}
\label{alg:composite-minimization}
\begin{algorithmic}[1]
  \STATE $x^0 \in \mathbb{R}^n$ such that $Ax^0 = 0$
  \FOR{$k \geq 0$}
  	\STATE Select a random edge $(i_k,j_k) \in E$ with probability $p_{i_kj_k}$
  	\STATE $d^k \leftarrow U_{i_kj_k} Z(f,x^k,(i_k,j_k),\alpha_{k}/L_{i_kj_k})$
  	\STATE $x^{k+1} \leftarrow x^k + d^k$
  	\STATE $k \leftarrow k + 1$
  \ENDFOR
\end{algorithmic}
\end{algorithm}

\subsection{Asynchronous Parallel Algorithm for Smooth Minimization}
\label{sec:async}
Although the algorithm described in the previous section solves a simple subproblem at each iteration, it is inherently sequential. This can be a disadvantage when addressing large-scale problems. To overcome this concern, we develop an asynchronous parallel method that solves Problem~\eqref{eq:opt} for the smooth case. 

Our parallel algorithm is similar to Algorithm \ref{alg:composite-minimization}, except for a crucial difference: now we may have multiple processors, and each of these executes the loop 2--6 \emph{independently} without the need for coordination. This way, we can solve subproblems (i.e., multiple pairs) simultaneously in parallel, and due to the asynchronous nature of our algorithm, we can execute updates as they complete, without requiring any locking. 

The critical issue, however, with implementing an asynchronous algorithm 
in the presence of non-separable constraints is ensuring feasibility throughout the course of the algorithm. This requires the operation $x_{i} \gets x_i + \delta$ to be executed in an \emph{atomic} (i.e., sequentially consistent) fashion.
Modern processors facilitate that without an additional locking structure through the ``compare-and-swap'' instruction~\citep{Recht11}.
Since the updates use atomic increments and each update satisfies 
$A d^k = 0$, the net effect of $T$ updates is $\sum_{k=1}^T A d^k = 0$, which is feasible despite asynchronicity of the algorithm.

The next key issue is that of convergence. In an asynchronous setting, the updates are based on \emph{stale} gradients that are computed using values of $x$ read many iterations earlier. But provided that gradient staleness is bounded, we can establish a sublinear convergence rate of the asynchronous parallel algorithm (Theorem \ref{thm:asyn-conv}). 
More formally, we assume that in iteration $k$, \emph{stale} gradients are computed based on $x^{D(k)}$
such that $k - D(k) \leq \tau$. 
The bound on staleness, denoted by $\tau$, captures the degree of parallelism in the method: such parameters are typical in asynchronous systems and provides a bound on the delay of the updates~\citep{Liu13}.

Before concluding the discussion on our asynchronous algorithm, it is important to note the difficulty of extending our algorithm to nonsmooth problems. For example, consider the case where $h = \mathbb{I}_C$ (indicator function of some convex set). Although a pairwise update as suggested above maintains feasibility with respect to the linear constraint $Ax = 0$, it may violate the feasibility of being in the convex set $C$. This complication can be circumvented by using a convex combination of the current iterate with the update, as this would retain overall feasibility. However, it would complicate the convergence analysis. We plan to investigate this direction in future work.

\subsection{Stochastic Minimization}
\label{sec:stoch}
An important subclass of Problem~\eqref{eq:opt} assumes separable losses $f(x) = \frac{1}{N} \sum_{i=1}^N f_i(x)$. This class arises naturally in many machine learning applications where the loss separates over training examples. To take advantage of this added separability of $f$, we can derive a stochastic block-CD procedure. 

Our key innovation here is the following: in addition to randomly picking an edge $(i,j)$, we also pick a function randomly from $\{f_1, \cdots, f_N\}$ and perform our update using this function. This choice substantially reduces the cost of each iteration when $N$ is large, since now the gradient calculations involve only the randomly selected function $f_i$ (i.e., we now use a stochastic-gradient). Pseudocode is given in Algorithm~\ref{alg:stochastic-minimization}. 

\begin{algorithm}
\caption{{\small Stochastic Minimization with Linear Constraints}}
\label{alg:stochastic-minimization}
\begin{algorithmic}[1]
  \STATE Choose $x^0 \in \mathbb{R}^n$ such that $Ax^0 = 0$.
  \FOR{$k \geq 0$}
  	\STATE Select a random edge $(i_k,j_k) \in E$ with probability $p_{i_kj_k}$
  	\STATE Select random integer $l \in [N]$
  	\STATE $x^{k+1} \leftarrow x^k + U_{i_kj_k} Z(f_l,x^k,(i_k,j_k),\alpha_{k}/L_{i_kj_k})$
  	\STATE $k \leftarrow k + 1$
  \ENDFOR
\end{algorithmic}
\end{algorithm}
Notice that the per iteration cost of Algorithm~\ref{alg:stochastic-minimization} is lower than  Algorithm~\ref{alg:composite-minimization} by a factor of $N$. However, as we will see later, this speedup comes at a price of slower convergence rate (Theorem~\ref{thm:stochastic-conv}). Moreover, to ensure convergence, decaying step sizes $\{\alpha_k\}_{k \geq 0}$ are generally chosen. 

%
%

\vspace*{-8pt}
\section{CONVERGENCE ANALYSIS}
\label{sec:conv}
\vspace*{-8pt}
In this section, we outline convergence results for the algorithms described above. The proofs are somewhat technical, and hence left in the appendix due to lack of space; here we present only the key ideas.

For simplicity, we present our analysis for the following reformulation of the main problem:
\begin{align}
\label{eq:reduced-optimization-problem}
 \min_{y,z}\quad  & f(y,z)+ \nlsum_{i=1}^b h(y_i,z_i) \\
\nonumber \quad \text{ subject to }&\ \ \nlsum_{i=1}^b y_i = 0,
\end{align}
where $y_i \in \mathbb{R}^{n_y}$ and $z_i \in \mathbb{R}^{n_z}$.
Let $y = [y_1^{\top} \cdots y_b^{\top}]^{\top}$ and $z = [z_1^{\top} \cdots z_b^{\top}]^{\top}$.
We use $x$ to denote the concatenated vector $[y^\top z^\top]^\top$ and hence we assume (unless otherwise mentioned)
that the constraint matrix $A$ is defined as follows
\begin{align}
A\left( \begin{array}{c} y \\ z \end{array} \right) = \left( \begin{array}{c} \sum_{i=1}^b y_i \\ 0 \end{array} \right).
\label{eq:a}
\end{align}
It is worth emphasizing that this analysis \emph{does not} result in any loss of generality.
This is due to the fact that Problem~(\ref{eq:opt}) with a general constraint matrix $\tilde{A}$ 
having full row-rank submatrices $\tilde{A}_i$'s
can be rewritten in the form of Problem~(\ref{eq:reduced-optimization-problem})
by using the transformation specified in Section~\ref{sec:reduction} of the appendix. It is important to note that this reduction is presented only for the ease of exposition. For our experiments, we directly solve the problem in Equation~\ref{eq:subproblem}.

Let $\eta_k = \{(i_0,j_0),\dots,(i_{k-1},j_{k-1})\}$ denote  the pairs selected up to iteration $k-1$. 
To simplify notation, assume (without loss of generality) that the Lipschitz constant for the partial gradient $\nabla_i f(x)$ and $\nabla_{ij} f(x)$ is $L$ for all $i \in [n]$ and $(i,j) \in E$.

Similar to \cite{necoara:11}, we introduce a Laplacian matrix ${\cal L} \in \mathbb{R}^{b \times b}$
that represents the communication graph $G$. Since we also have unconstrained variables $z_i$,
we introduce a diagonal matrix ${\cal D} \in \mathbb{R}^{b \times b}$.
\begin{align*}
{\cal L}_{ij} = \left\{
\begin{array}{ll}
\sum_{r \neq i} \frac{p_{ir}}{2L} & i = j\\
-\frac{p_{ij}}{2L} & i \neq j
\end{array}
\right.
& \hfill &
{\cal D}_{ij} = \left\{
\begin{array}{ll}
\frac{p_i}{L} & i = j\\
0 & i \neq j
\end{array}
\right.
\end{align*}

We use $\mathcal{K}$ to denote the concatenation of the Laplacian $\mathcal{L}$ and the diagonal matrix $\mathcal{D}$.
More formally,
$$
\mathcal{K} = \left[
\begin{array}{cc}
{\cal L} \otimes I_{n_y} & 0 \\ 0 & {\cal D} \otimes I_{n_z}
\end{array}\right].
$$
This matrix induces a norm $\|x\|_{\cal K} = \sqrt{x^\top \mathcal{K} x}$ on the \emph{feasible subspace}, with a corresponding dual norm
\begin{align*}
\|x\|_{{\cal K}}^* = \sqrt{x^\top 
\left(\left[
\begin{array}{cc}
{\cal L^+} \otimes I_{n_y} & 0 \\ 0 & {\cal D}^{-1} \otimes I_{n_z}
\end{array}\right]\right) x}
\end{align*}

Let ${X^*}$ denote the set of optimal solutions and let $x^0$ denote the initial point. We define the following distance, which quantifies how far the initial point is from the optimal, taking into account the graph layout and edge  selection probabilities
\begin{align}
\label{eq:r-def}
R(x^0) := \max_{x:f(x) \leq f(x^0)} \max_{x^* \in X^*}
	\left\| x - x^* \right\|_{{\cal K}}^*
\end{align}

\textbf{Note.} Before delving into the details of the convergence results, we would like to draw the reader's attention to the impact of the communication network $G$ on convergence. In general, the convergence results depend on $R(x^0)$, which in turn depends on the Laplacian $\mathcal{L}$ of the graph $G$.  As a rule of thumb, the larger the connectivity of the graph, the smaller the value of $R(x^0)$, and hence, faster the convergence.

\subsection{Convergence results for the smooth case}
We first consider the case when $h = 0$. Here the subproblem at $k^{\text{th}}$ iteration has a very simple update $d_{i_kj_k} = U_{i_k} d^k - U_{j_k} d^k$ where $d^k = \frac{\alpha_k}{2L}(\nabla_{j_k} f(x^k) - \nabla_{i_k} f(x^k))$. 
We now prove that Algorithm~\ref{alg:composite-minimization} attains an $O(1/k)$ convergence rate.

\begin{theorem}
Let $\alpha_k = 1$ for $k \geq 0$, and let $\{x^k\}_{k \geq 0}$ be the sequence generated by Algorithm~\ref{alg:composite-minimization}; let $f^*$ denote the optimal value.  Then, we have the following rate of convergence:
\begin{align*}
\mathbb{E}[f(x^k)]  - f^*  \leq \frac{2R^2(x^0)}{k}
\end{align*}
where $R(x^0)$ is as defined in Equation~\ref{eq:r-def}.
\label{thm:conv}
\end{theorem}
\begin{proof}[Proof Sketch]
We first prove that each iteration leads to descent in expectation. More formally, we get
\begin{align*}
\mathbb{E}_{i_kj_k}[f(x^{k+1})|\eta_k] \leq f(x^k) -\tfrac{1}{2} \nabla f(x^k)^\top \mathcal{K} \nabla f(x^k).
\end{align*}
The above step can be proved using Lemma~\ref{lem:descent-lemma}. Let $\Delta_k = \mathbb{E}[f(x^k)] - f^*$. It can be proved that
\begin{align*}
\frac{1}{\Delta_k} \leq \frac{1}{\Delta_{k+1}} - \frac{1}{2R^2(x^0)}
\end{align*}
This follows from the fact that 
\begin{align*}
f(x^{k+1}) - f^* & \leq \|x^k - x^*\|_{\cal K}^* \| \nabla f(x^k) \|_{\cal K} \\
& \leq R(x^0)\| \nabla f(x^k) \|_{\cal K} \quad \forall k \geq 0
\end{align*}
Telescoping the sum, we get the desired result.
\end{proof}

Note that Theorem~\ref{thm:conv} is a strict generalization of the analysis in \citep{necoara:11} and \citep{Necoara14} due to: (i) the presence of unconstrained variables $z$; and  (ii) the presence of a non-decomposable objective function. it is also worth emphasizing that our convergence rates improve upon those of \citep{Necoara14}, since they do not involve an exponential dependence of the form $b^m$ on the number of constraints. 

We now turn our attention towards the convergence analysis of our asynchronous algorithm
under a consistent reading model \citep{Liu13}. In this context we would like to emphasize that while our theoretical analysis 
assumes consistent reads, we do not enforce this assumption in our experiments.

\begin{theorem}
Let $\rho > 1$ and $\alpha_k = \alpha$ be such that $\alpha < 2/(1+\tau+\tau \rho^\tau)$ and 
$\alpha < (\rho - 1)/(\sqrt{2}(\tau+2)(\rho^{\tau+1}+\rho))$. Let $\{x_k\}_{k \geq 0}$ be the sequence generated by asynchronous algorithm using step size $\alpha_k$ and let $f^*$ denote the optimal value. Then, we have the following rate of convergence for the expected values of the objective function
\begin{align*}
\mathbb{E}[f(x_k)]  - f^*  \leq \frac{R^2(x^0)}{\mu k}
\end{align*}
where $R(x^0)$ is as defined in Equation~\ref{eq:r-def} and $
\mu = \frac{\alpha_k^2}{2}\left(\frac{1}{\alpha_k} - \frac{1+\tau+\tau\rho^\tau}{2}\right)
$.
\label{thm:asyn-conv}
\end{theorem}
\begin{proof}[Proof Sketch] For ease of exposition, we describe the case where the unconstrained variables $z$ are absent. The analysis of case with $z$ variables can be carried out in a similar manner. Let $D(k)$ denote the iterate of the variables used in the $k^{\text{th}}$ iteration (the existence of $D(k)$ follows from the consistent reading assumption). Let $d^k = \frac{\alpha_k}{2L} \left(\nabla_{y_{j_k}} f(x^{D(k)}) - \nabla_{y_{i_k}} f(x^{D(k)}) \right)$ and $d_{i_kj_k}^k = x^{k+1} - x^{k} = U_{i_k}d^k - U_{j_k}d^k$. Using Lemma~\ref{lem:descent-lemma} and the assumption that staleness  in the variables is bounded by $\tau$, i.e., $k - D(k) \leq \tau$ and definition of $d_{ij}^k$, we can derive the following bound:
\begin{align*} 
\mathbb{E}[f(&x^{k+1})] \leq \mathbb{E}[f(x^k)] - L \left(\frac{1}{\alpha_k} - \frac{1+\tau}{2}\right) \mathbb{E}[\|d_{i_kj_k}^k\|^2] \\
& \quad \quad \quad  \quad + \frac{L}{2} \mathbb{E}\left[\sum_{t=1}^{\tau} \|d_{i_{k-t}j_{k-t}}^{k-t}\|^2\right].
\end{align*}
In order to obtain an upper bound on the norms of $d^k_{i_kj_k}$, we prove that 
\begin{align*}
\mathbb{E}\left[ \|d_{i_{k-1}j_{k-1}}^{k-1}\|^2\right] \leq \rho \mathbb{E}\left[ \|d_{i_{k}j_{k}}^k\|^2\right]
\end{align*}
This can proven using mathematical induction. Using the above bound on $\|d_{i_{k}j_{k}}^k\|^2$, we get 
\begin{align*}
& \mathbb{E}[f(x^{k+1})] \leq \\
& \quad \mathbb{E}[f(x^k)] - L \left(\frac{1}{\alpha_k} - \frac{1+\tau + \tau\rho^{\tau}}{2}\right) \mathbb{E}[\|d_{i_kj_k}^k\|^2]
\end{align*}
This proves that the method is a descent method in expectation. Following similar analysis as Theorem~\ref{thm:conv}, we get the required result.
\end{proof}

Note the dependence of convergence rate on the staleness bound $\tau$. For larger values of $\tau$, the stepsize $\alpha_k$ needs to be decreased to ensure convergence, which in turn slows down the convergence rate of the algorithm. Nevertheless, the convergence rate remains $O(1/k)$. 

The last smooth case we analyze is our stochastic algorithm.
\begin{theorem}
\label{thm:stoch}
Let $\alpha_i = \sqrt{\Delta_0 L}/(M\sqrt{i+1})$ for $i \geq 0$ in Algorithm~\ref{alg:stochastic-minimization}. Let $\{x^k\}_{k \geq 0}$ be the sequence generated by Algorithm~\ref{alg:stochastic-minimization} and let $f^*$ denote the optimal value. We denote $\bar{x}^k = \arg\min_{0 \leq i \leq k} f(x^k)$. Then, we have the following rate of convergence for the expected values of the objective:
\begin{align*}
\mathbb{E}[f(\bar{x}^k)]  - f^*  \leq O\left(\frac{1}{\sqrt[4] k}\right)
\end{align*}
where $\Delta_0 = f(x^0) - f^*$.
\label{thm:stochastic-conv}
\end{theorem}
The convergence rate is $O(1/k^{1/4})$ as opposed to $O(1/k)$ of Theorem~\ref{thm:conv}. 
On the other hand, the iteration complexity is lower by a factor of $N$; this kind of tradeoff is typical in stochastic algorithms,
where the slower rate is the price we pay for a lower iteration complexity. 
We believe that the convergence rate can be improved to $O(1/\sqrt{k})$, the rate generally observed in stochastic algorithms, by a more careful analysis.

\subsection{Nonsmooth case}
We finally state the convergence rate for the nonsmooth case ($h\not\equiv0$) in the case of a sum constraint. 
Similar to \cite{Necoara14}, we assume $h$ is coordinatewise separable (i.e. 
we can write $h(x) = \sum_{i=1}^b \sum_{j} x_{ij}$), where $x_{ij}$
is the $j^{th}$ coordinate in the $i^{th}$ block.
For this analysis, we assume that the graph $G$ is a clique 
\footnote{We believe our results also easily extend to the
general case along the lines of \cite{richTak11,richTak12,richtarik:13}, using the concept of
\emph{Expected Separable Overapproximation} (ESO). Moreover, the assumption is not totally impractical,
e.g., in a multicore setting with a zero-sum constraint (i.e. $A_i = I$), the clique-assumption introduces little cost.} with uniform probability, i.e., $\lambda = p_{ij} = 2/b(b-1)$.

\begin{theorem}
\label{thm:non-smooth-conv}
Assume $Ax = \sum_i A_i x_i$. 
Let $\{x^k\}_{k \geq 0}$ be the sequence generated by Algorithm~\ref{alg:composite-minimization} and let $F^*$ denote the optimal value. Assume that the graph $G$ is a clique with uniform probability. Then we have the following:
$$
\mathbb{E}[F(x^k) - F^*] \leq \frac{b^2 L R^2(x^0)}{2k + \frac{b^2 L R^2(x^0)}{\Delta_0}},
$$
where $R(x^0)$ is as defined in Equation~\ref{eq:r-def}.
\end{theorem}
This convergence rate is a generalization of the convergence rate
obtained in \citet{Necoara14} for a single linear constraint (see Theorem 1 in \citep{Necoara14}).
It is also an improvement of the rate obtained in \citet{Necoara14} for general linear constraints (see Theorem 4 in \citep{Necoara14}) when applied to the special case of a sum constraint. 
Our improvement comes in the form of a tractable constant, as opposed to the exponential dependence $O(b^m)$ shown in~\citep{Necoara14}.

\vspace*{-8pt}
\section{APPLICATIONS}
\vspace*{-6pt}
\input{applications}

\section{EXPERIMENTS}
\input{experiments}

\section{DISCUSSION AND FUTURE WORK}
\vspace{-2 mm}
We presented randomized coordinate descent methods for solving convex optimization problems with linear constraints that couple the variables. 
Moreover, we also presented composite objective, stochastic, and asynchronous versions of our basic method 
and provided their convergence analysis.
We demonstrated the empirical performance of the algorithms. The experimental results of asynchronous algorithm look very promising. 

There are interesting open problems for our problem in consideration: First, we would like to obtain high-probability results not just in expectation; another interesting direction is to extend the asynchronous algorithm to the non-smooth setting. Finally, while we obtain $O(1/k)$ for general convex functions, obtaining an accelerated $O(1/k^2)$ rate is a natural question.

\subsubsection*{Acknowledgments}
SS is partly supported by NSF grant: IIS-1409802. 
We thanks the anonymous reviewers for the helpful comments.

{
\bibliographystyle{abbrvnat}
\setlength{\bibsep}{1pt}
\bibliography{bibfile}
}

\clearpage
\appendix
\onecolumn
\section*{Appendix}
\input{appendix}

\end{document}

%% file: header.tex
\usepackage{etoolbox}
\newbool{tikC}
\setbool{tikC}{false} 

\pdfoutput=1                    
\usepackage[utf8]{inputenc}

\usepackage{tikz}
\usepackage{pgfplots}

\usepackage{graphicx}
\usepackage{amsmath,amssymb,amsthm,bm} 
\usepackage{mathrsfs}

\usepackage[sort&compress]{natbib}
\usepackage[vlined,boxed,commentsnumbered]{algorithm2e}
\usepackage{hyperref}
\usepackage{subfigure} 
\usepackage{xspace}
\usepackage{array}
\usepackage{enumerate}
\usepackage{paralist}

\bibpunct[; ]{[}{]}{,}{n}{}{;}

\newtheorem{theorem}{Theorem}
\newtheorem{lemma}[theorem]{Lemma}

\theoremstyle{definition}

\newcommand{\reals}{\mathbb{R}}
\newcommand{\nlsum}{\sum\nolimits}

\newcommand{\nlmin}{\min\nolimits}
\newcommand{\norm}[1]{\|#1\|}
\newcommand{\pnorm}[2]{\|#1\|_{#2}}



%% file: applications.tex
To gain a better understanding of our approach, we state some applications of interest, while discussing details of Algorithm~\ref{alg:composite-minimization} and Algorithm~\ref{alg:stochastic-minimization} for them. While there are many applications of problem~\eqref{eq:opt}, due to lack of space we only mention a few prominent ones here.

{\bf Support Vector Machines:}
The SVM dual (with bias term) assumes the form~(\ref{eq:opt}); specifically, 
\begin{eqnarray}
\nonumber \min_{\alpha} \tfrac{1}{2}\nlsum_{i,j}\alpha_i \alpha_j y_i y_j {z}_i^\top  z_j - \nlsum_{i=1}^n \alpha_i \ \ \\
s.t. \nlsum_i \alpha_iy_i = 0, \ \ 0 \leq \alpha_i \leq C \quad \forall\ i \in [n].
\label{eq:svm}
\end{eqnarray}
Here, $z_i$ denotes the feature vector of the $i^{th}$ training example
and $y_i \in \{1,-1\}$ denotes the corresponding label.
By letting $f(\alpha) = \frac{1}{2}\sum_{i,j}\alpha_i \alpha_j y_i y_j z_i^\top z_j - \sum_i\alpha_i$ and $h(\alpha) = \sum_i \mathbb{I}(0\leq \alpha_i \leq C)$ and $A = [y_1,\dots,y_n]$ this problem can be written in form of Problem~(\ref{eq:opt}). Using Algorithm~\ref{alg:composite-minimization} for SVM involves solving a sub-problem similar to one used in SMO in the scalar case (i.e., $\alpha_i \in \mathbb{R}$) and can be solved in linear time in the block case (see \citep{Pardalos}).

{\bf Generalized Lasso:}
The objective is to solve the following optimization problem. 
\begin{eqnarray*}
\min\nolimits_{\beta}\quad \tfrac{1}{2}\|Y - X\beta\|_2^2	+ \lambda \|D\beta\|_1
\end{eqnarray*}
where $Y\in \reals^N$ denotes the output, $X \in \reals^{N\times n}$ is the input and $D \in \reals^{q\times n}$ represents a specified penalty matrix. This problem can also be seen as a specific case of Problem~(\ref{eq:opt}) by introducing an auxiliary variable $t$ and slack variables $u,v$. Then, $f(\beta,t) = \frac{1}{2}\|Y - X\beta\|_2^2 + \sum_i t_i$, $h(u,v) = \mathbb{I}(u \geq 0) + \mathbb{I}(v \geq 0)$ and, $t - D\beta - u = 0$  and $t + D\beta - v = 0$ are the linear constraints. To solve this problem, we can use either Algorithm~\ref{alg:composite-minimization} or Algorithm~\ref{alg:stochastic-minimization}. In general, optimization of convex functions on a structured convex polytope can be solved in a similar manner.

{\bf Unconstrained Separable Optimization:} Another interesting application is for unconstrained separable optimization. For any problem $\min_x \sum_i {f_i(x)}$---a form generally encountered across machine learning---can be rewritten using variable-splitting as $\min_{\{x_i = x, \forall i \in [N]\}} f_i(x_i)$. Solving the problem in distributed environment requires considerable synchronization (for the consensus constraint), which can slow down the algorithm significantly. However, the dual of the problem is
\begin{align*}
\min_{\lambda} \sum_i {f_i^*(\lambda_i)} \quad 
s.t \ \ \nlsum_{i=1}^N \lambda_i = 0.
\end{align*}
where $f_i^*$ is the Fenchel conjugate of $f_i$. This reformulation perfectly fits our framework and can be solved in an asynchronous manner using the procedure described in Section~\ref{sec:async}.

Other interesting application include constrained least square problem, multi-agent planning problems, resource allocation---see \cite{necoara:11,Necoara14} and references therein for more examples.


%% file: experiments.tex
In this section, we present our empirical results. In particular, we examine the behavior of random coordinate descent 
algorithms analyzed in this paper under different communication constraints and concurrency conditions.
\footnote{All experiments were conducted on a Google Compute Engine
virtual machine of type ``n1-highcpu-16'', which comprises 16 virtual CPUs and 14.4 GB of memory.
For more details, please refer to \url{https://cloud.google.com/compute/docs/machine-types\#highcpu}.
}

\vspace*{-8pt}
\subsection{Effect of Communication Constraints}
\vspace*{-8pt}

Our first set of experiments test the affect of the connectivity of the graph on the convergence rate. In particular, recall that the convergence analysis established in  Theorem~\ref{thm:conv} depends on the Laplacian of the communication graph. In this experiment we demonstrate how communication constraints affect convergence in practice.
We experiment with the following graph topologies of graph $G$: 
Ring, Clique, Star + Ring (i.e., the union of edges of a star and a ring) and Tree + Ring. On each layout we run the sequential Algorithm~\ref{alg:composite-minimization} on the following
quadratic problem
\begin{align}
\nonumber \min & \quad C \nlsum_{i=1}^N \|x_i - (i\ \mathrm{mod}\ 10) \bm{1}\|^2 \ \ \\
& s.t. \ \ \nlsum_{i=1}^N A_ix_i = 0,
\label{eq:synth}
\end{align}

Note the decomposable structure of the problem. For this experiment, we use $N = 1000$ and $x_i \in \mathbb{R}^{50}$. 
We have 10 constraints whose coefficients are randomly generated from $U[0,1]$
and we choose $C$ such that the objective evaluates to 1000 when $x = 0$.

\begin{figure}{h}
  \centering
  \includegraphics[width=5.5cm,height=4cm]{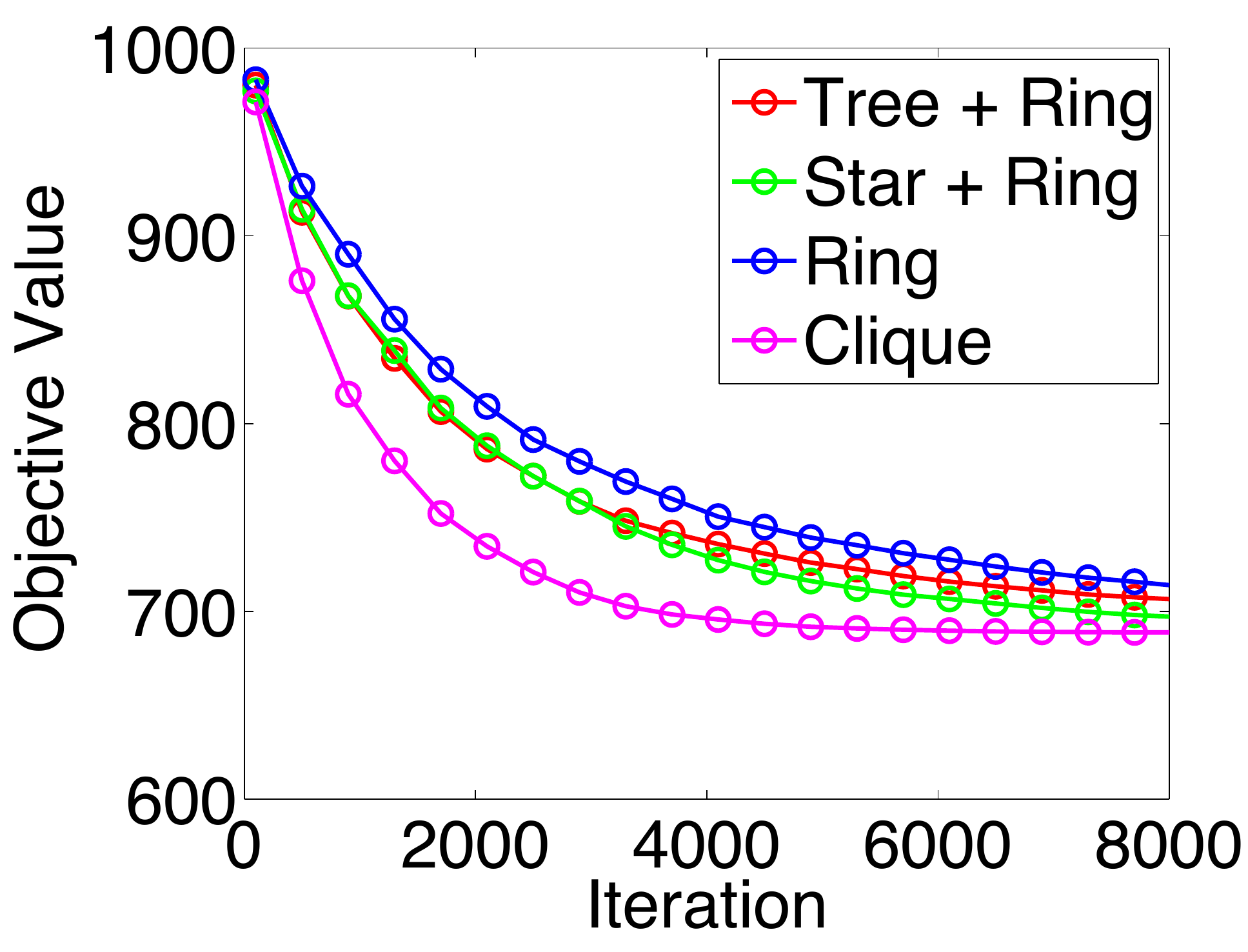}
  \vspace*{-8pt}
  \caption{\small Objective value vs. number of iterations for different graph topologies. Note that larger the connectivity of the graph, faster is the convergence.}
  \label{fig:layout}
\end{figure}

The results for Algorithm~\ref{alg:composite-minimization} on each topology for 10000 iterations are shown in Figure~\ref{fig:layout}. The results clearly show that better connectivity implies better convergence rate. 
Note that while the clique topology has significantly better convergence than other topologies, acceptable long-term  performance can be achieved by much sparser topologies such as Star + Ring and Tree + Ring.

Having a sparse communication graph is important to lower the cost of a distributed system. Furthermore, it is worth mentioning that the sparsity of the communication graph is also important in a multicore setting;  since Algorithm \ref{alg:composite-minimization} requires computing $(A_i A_i^\top + A_j A_j^\top)^+$ for each communicating pair of nodes ($i$, $j$). Our analysis shows that this computation takes a significant portion of the running time and hence it is essential to minimize the  number of variable pairs that are allowed to be updated.

\subsection{Concurrency and Synchronization}
\label{sec:sync}

As seen earlier, compared to Tree + Ring, Star + Ring is a low diameter layout (diameter = 2). Hence, in a sequential setting, it indeed results in a faster convergence. However, Star + Ring requires a node to be connected to all other nodes. This high-degree node could be a contention point in a parallel setting. We test the performance of our asynchronous algorithm in this setting. To assess how the performance would be affected with such contention and how asynchronous updates would increase performance, we conduct
another experiment on the synthetic problem \eqref{eq:synth}
but on a larger scale ($N = 10000$, $x_i \in \mathbb{R}^{100}$, 100 constraints).

Our concurrent update follows a master/slave scheme.
Each thread performs a loop where in each iteration it
elects a master $i$ and slave $j$ and then applies the following sequence of actions:
\begin{enumerate}
\item Obtain the information required for the update from the master (i.e., information for calculating the gradients used for solving the subproblem).
\item Send the master information to the slave, update the slave variable and get back the information needed to update the master.
\item Update the master based (only) on the information obtained from steps 1 and 2.
\end{enumerate}
We emphasize that the master is not allowed to read its own state at step 3 except to apply an increment, which is computed based on steps 1 and 2.
This ensures that the master's increment is consistent with that of the slave,
even if one or both of them was being concurrently overwritten by another thread.
More details on the implementation can be found in \cite{nipsws}.

Given this update scheme, we experiment with three levels of synchronization:
(a) \textbf{Double Locking:} Locks the master and the slave through the entire update. Because the objective function is decomposable, a more conservative locking (e.g. locking all nodes) is not needed. (b) \textbf{Single Locking:} Locks the master during steps 1 and 3  (the master is unlocked during step 2 and locks the slave during step 2). (c) \textbf{Lock-free:} No locks are used. Master and slave variables are updated through atomic increments similar to Hogwild! method.

Following \cite{Recht11}, we use spinlocks instead of mutex locks to implement locking. Spinlocks are preferred over mutex locks when the resource is locked for a short period of time, which is the case in our algorithm. For each locking mechanism, we vary the number of threads from 1 to 15. 
We stop when $f_0 - f_t > 0.99 (f_0 - f^*)$, where $f^*$ is computed beforehand up to three significant digits. Similar to \cite{Recht11}, we add artificial delay to steps 1 and 2 in the update scheme to model complicated gradient calculations and/or network latency in a distributed setting. 

Figure \ref{fig:speedup} shows the speedup for Tree + Ring and Star + Ring layouts.
The figure clearly shows that a fully synchronous method suffers from contention in the Star + Ring topology whereas asynchronous method does not suffer from this problem and hence, achieves higher speedups. Although the Tree + Ring layouts achieves higher speedup than Star + Ring, the latter topology results in much less running time ($\sim$ 67 seconds vs 91 seconds using 15 threads).

\begin{figure}
\begin{center}
\begin{tabular}{cc}
\includegraphics[width=0.8\linewidth]{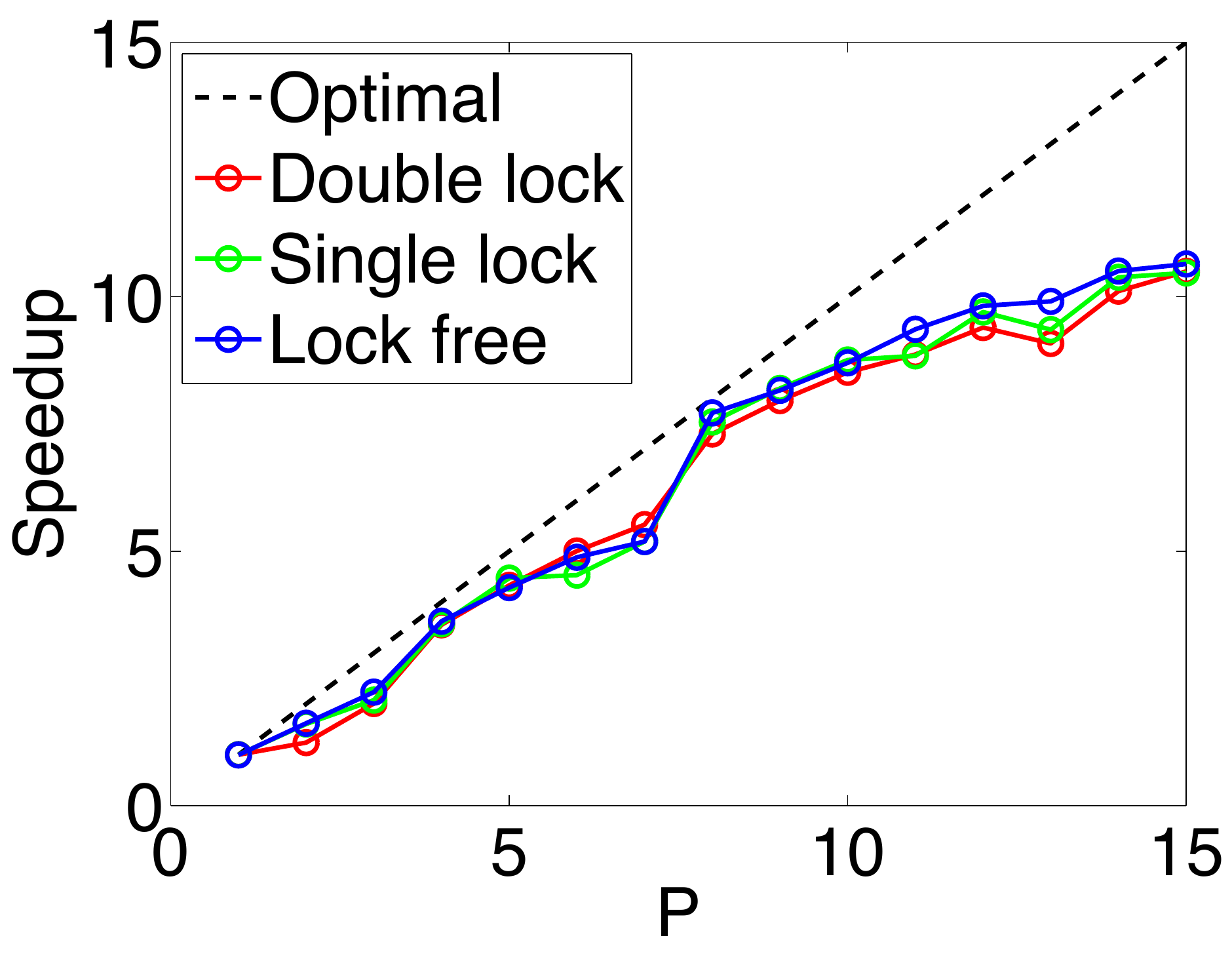}\\
\includegraphics[width=0.8\linewidth]{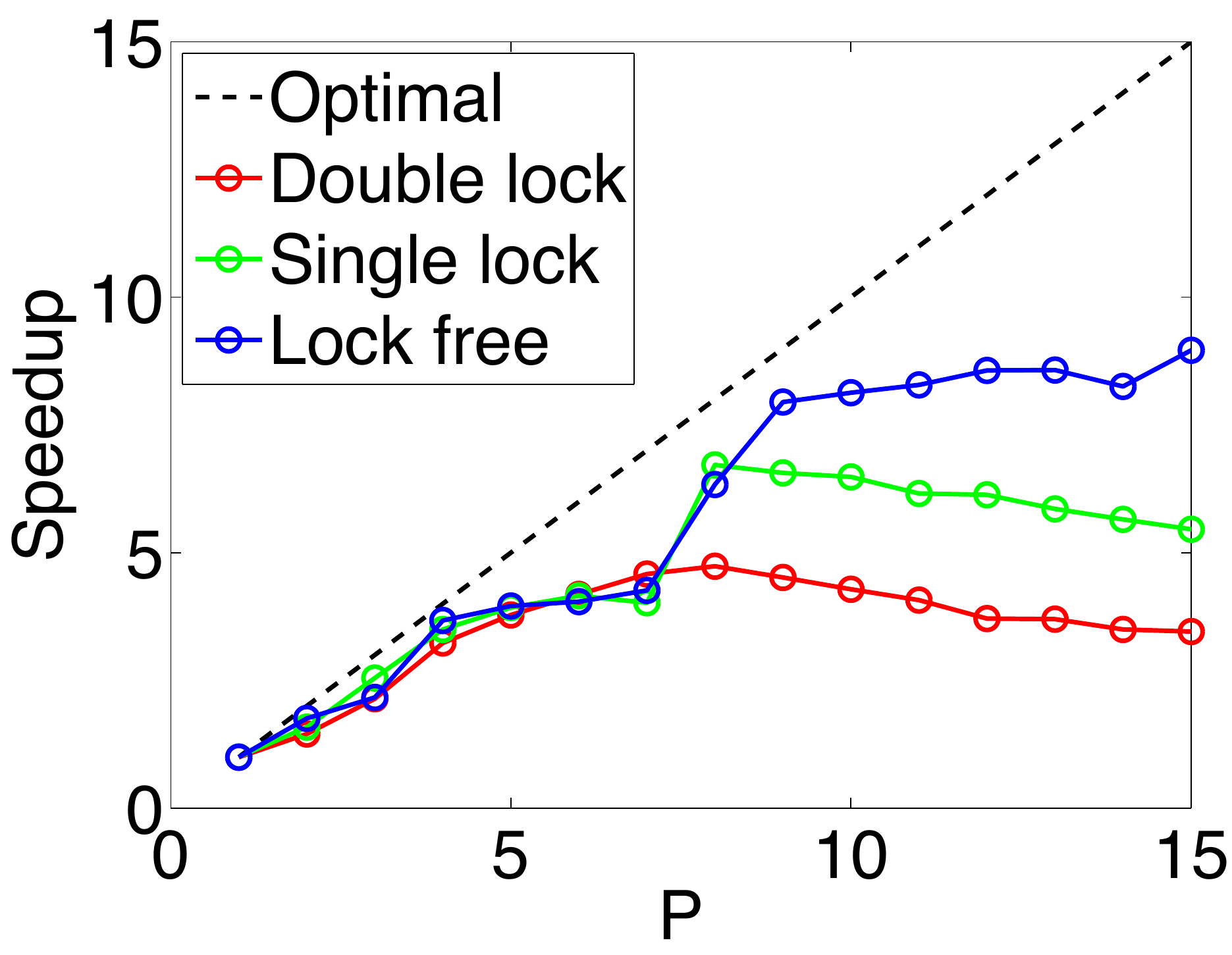}
\end{tabular}
\caption{Speedup for Tree + Ring (top) and Star + Ring (bottom) topologies and different levels of synchronization. Note for Star + Ring topology, speedup of asynchronous algorithm is significantly higher than that of synchronous version.}
\label{fig:speedup}
\end{center}
\end{figure}

\subsection{Practical Case Study: Parallel Training of Linear SVM}
In this section, we explore the effect of parallelism on randomized CD for training a linear SVM
based on the dual formulation stated in \eqref{eq:svm}.
Necoara et. al. \citep{Necoara14} have shown that, in terms of CPU time, a sequential randomized CD outperforms coordinate descent
using Gauss-Southwell selection rule.
It was also observed that randomized CD outperforms LIBSVM \cite{libsvm}
for large datasets while maintaining reasonable performance for small datasets.

In this experiment we use a clique layout. For SVM training in a multicore setting, using a clique layout 
does not introduce additional cost compared to a more sparse layout. 
To maintain the box constraint, we use the double-locking scheme described in Section \ref{sec:sync}
for updating a pair of dual variables.

One advantage of coordinate descent algorithms is that they do not require the storage of the Gram matrix;
instead they can compute its elements on the fly. That comes, however, at the expense of CPU time.
Similar to \cite{Necoara14}, to speed up gradient computations without increasing memory requirements,
we maintain the primal weight vector
of the linear SVM and use it to compute gradients. Basically, if we increment $\alpha_i$ by $\delta_i$
and $\alpha_j$ by $\delta_j$, then we increment the weight vector by $\delta_i y_i x_i + \delta_j y_j x_j$.
This increment is accomplished using atomic additions. 
However, this implies that all threads will be concurrently updating the primal weight vector.
Similar to \cite{Recht11}, we require these updates
to be sparse with small overlap between non-zero coordinates
in order to ensure convergence. 
In other words, we require training examples to have sparse features with small overlap between non-zero features.

We report speedups on two datasets used in \citep{Necoara14}.\footnote{Datasets can be downloaded from
\url{http://www.csie.ntu.edu.tw/~cjlin/libsvmtools/datasets}.
}
Table \ref{tbl:datasets} provides a description of both the datasets. 
For each dataset, we train the SVM model until $f_0 - f_t > 0.9999(f_0 - f^*)$, where $f^*$
is the objective reported in \citep{Necoara14}.
In Figure \ref{fig:speedup_svm}, we report speedup for both the datasets. The figure shows that parallelism indeed increases the performance of randomized CD training of linear SVM.

\begin{table}
\centering
\begin{tabular}{|x{1cm}|x{1.5cm}|x{1.2cm}|x{2cm}|}
\hline
Dataset & \# of instances & \# of features & Avg \# of non-zero features \\
\hline
{\tt a7a} & 16100 & 122 & 14\\
\hline
{\tt w8a} & 49749 & 300 & 12\\
\hline
\end{tabular}
\caption{Datasets used for linear SVM Speedup experiment}
\label{tbl:datasets}
\end{table}

\begin{figure}
\begin{center}
\begin{tabular}{cc}
\includegraphics[width=0.8\linewidth]{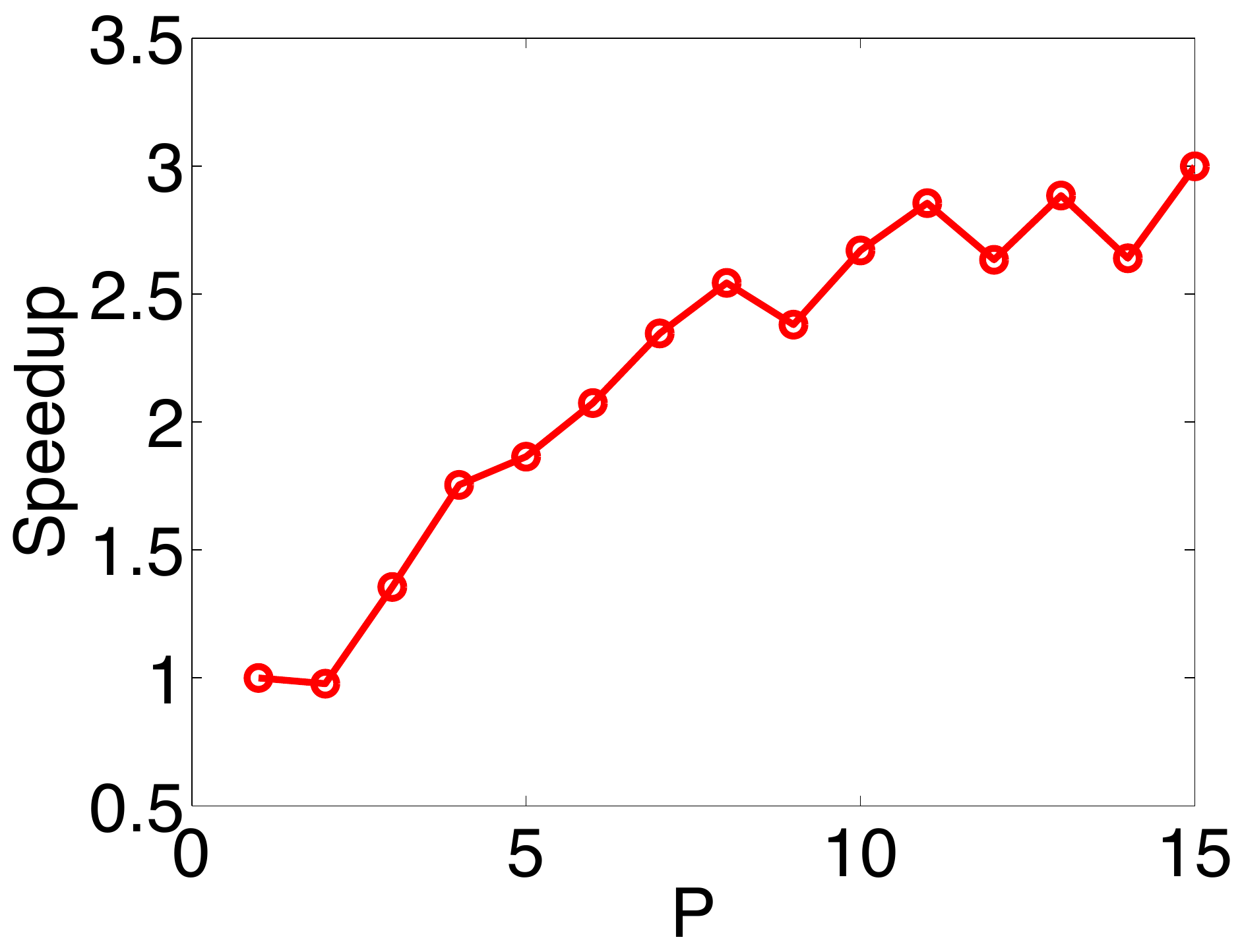}\\
\includegraphics[width=0.8\linewidth]{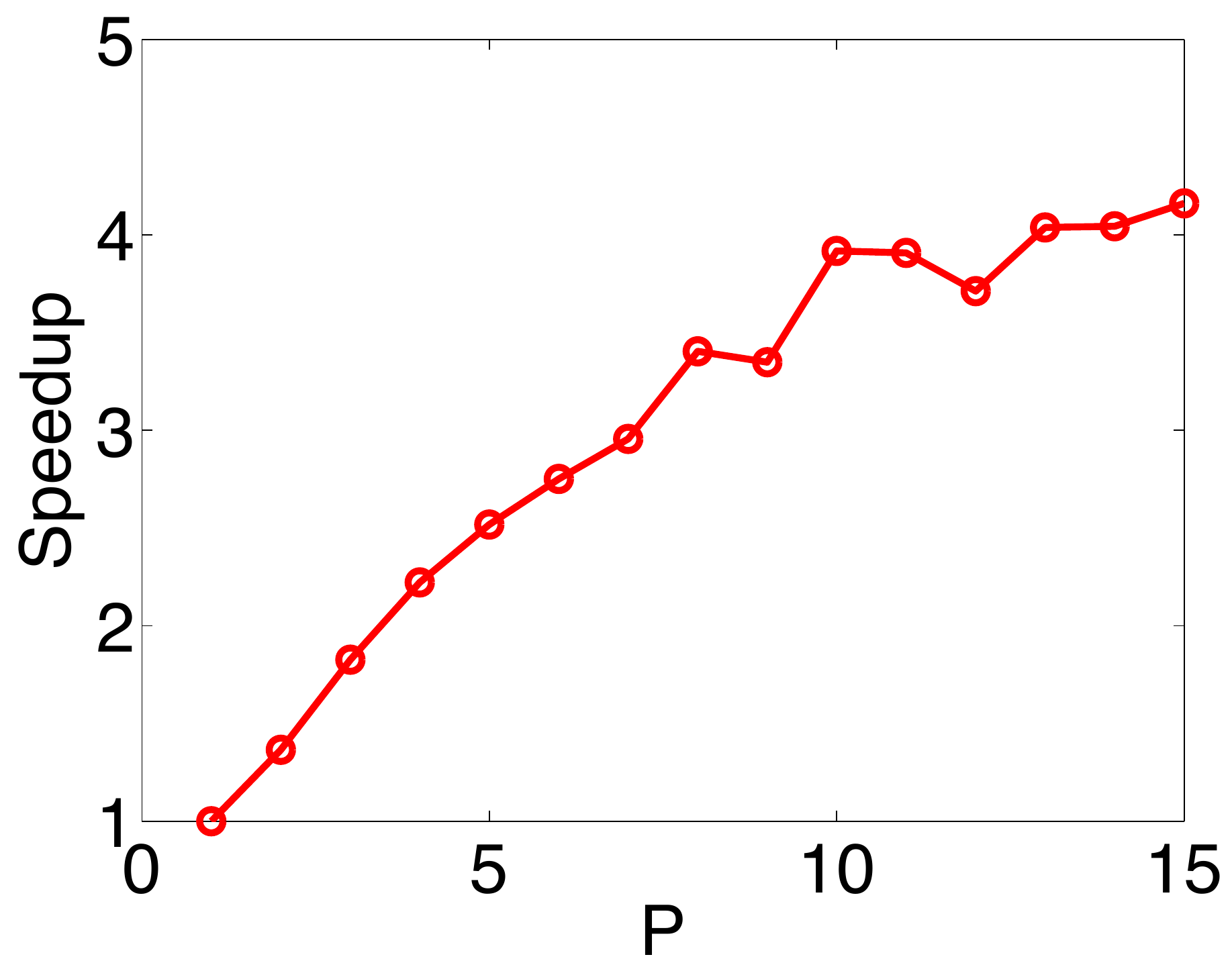}
\end{tabular}
\caption{Speedup for linear SVM training on {\tt a7a} (top) and {\tt w8a} datasets.}
\label{fig:speedup_svm}
\end{center}
\end{figure}


%% file: appendix.tex
\section{Proof of Theorem \ref{thm:conv}}
\begin{proof}
Taking the expectation over the choice of edges $(i_k,j_k)$ gives the following inequality
\begin{align}
\label{eq:conv-eq1}
\nonumber \mathbb{E}_{i_kj_k}[f(x^{k+1})|\eta_k] &\leq \mathbb{E}_{i_kj_k}\left[f(x^k)
-\frac{1}{4L} \|\nabla_{y_{i_k}} f(x^k) - \nabla_{y_{j_k}} f(x^k)\|^2
 -\frac{1}{2L} \|\nabla_{z_{i_k}} f(x^k)\|^2 - \frac{1}{2L}  \|\nabla_{z_{j_k}} f(x^k)\|^2\right] \\
\nonumber &\leq f(x^k)
 -\frac{1}{2}\nabla_y f(x^k)^\top ({\cal L} \otimes I_{n_y}) \nabla_y f(x^k)
  -\frac{1}{2} \nabla_z f(x^k)^\top ({\cal D} \otimes I_{n_z}) \nabla_z f(x^k) \\
&\leq f(x^k) -\frac{1}{2} \nabla f(x^k)^\top \mathcal{K} \nabla f(x^k),
\end{align}
where $\otimes$ denotes the Kronecker product. This shows that the method is a descent method. Now we are ready to prove the main convergence theorem. We have the following:
\begin{align*}
f(x^{k+1}) - f^* & \leq \langle \nabla f(x^k), x^k - x^* \rangle
\leq \|x^k - x^*\|_{\cal K}^* \| \nabla f(x^k) \|_{\cal K} \\
& \leq R(x^0)\| \nabla f(x^k) \|_{\cal K} \quad \forall k \geq 0.
\end{align*}
Combining this with inequality \eqref{eq:conv-eq1}, we obtain
\begin{align*}
\mathbb{E}[f(x^{k+1}|\eta_k] \leq  f(x^k) - \frac{(f(x^k)-f^*)^2}{2R^2(x^0)}.
\end{align*}
Taking the expectation of both sides an denoting $\Delta_k = \mathbb{E}[f(x^k)] - f^*$ gives
\begin{align*}
\Delta_{k+1} \leq \Delta_k - \frac{\Delta_k^2}{2R^2(x^0)}.
\end{align*}
Dividing both sides by $\Delta_k \Delta_{k+1}$ and using the fact that $\Delta_{k+1} \leq \Delta_k$ we obtain
\begin{align*}
\frac{1}{\Delta_k} \leq \frac{1}{\Delta_{k+1}} - \frac{1}{2R^2(x^0)}.
\end{align*}
Adding these inequalities for $k$ steps $0 \leq \frac{1}{\Delta_0} \leq \frac{1}{\Delta_k} - \frac{k}{2R^2(x^0)}$ from which we obtain the statement of the theorem where $C = 2R^2(x^0)$.
\end{proof}

\section{Proof of Theorem \ref{thm:stochastic-conv}}

\begin{proof}
In this case, the expectation should be over the selection of the pair $(i_k,j_k)$ and random index $l_k \in [N]$. In this proof, the definition of $\eta_k$ includes $l_k$ i.e., $\eta_k = \{(i_0,j_0,l_0), \dots, (i_{k-1},j_{k-1},l_{k-1})\}$. We define the following:
\begin{align*}
d_{i_k}^k &=  \left[\frac{\alpha_k}{2L} \left[\nabla_{y_{j_k}} f_{l_k}(x^{k}) - \nabla_{y_{i_k}} f_{l_k}(x^{k}) \right]^{\top} ,\quad -\frac{\alpha_k}{L} \left[\nabla_{z_{i_k}} f_{l_k}(x^{k}) \right]^{\top} \right]^{\top},\\
d_{j_k}^k &=  \left[\frac{\alpha_k}{2L} \left[\nabla_{y_{j_k}} f_{l_k}(x^{k}) - \nabla_{y_{i_k}} f_{l_k}(x^{k}) \right]^{\top} ,\quad \frac{\alpha_k}{L} \left[\nabla_{z_{j_k}} f_{l_k}(x^{k}) \right]^{\top} \right]^{\top},\\
d_{i_kj_k}^{l_k} &= U_{i_k} d_{i_k}^k - U_{j_k} d_{j_k}^k.
\end{align*}
For the expectation of objective value at $x^{k+1}$, we have
\begin{align*}
&\mathbb{E}[f(x^{k+1})|\eta_k] \leq \mathbb{E}_{i_kj_k}\mathbb{E}_{l_k}\left[f(x^k)
+ \left\langle \nabla f(x^k), d_{i_kj_k}^{l_k} \right\rangle + \frac{L}{2} \|d_{i_kj_k}^{l_k}\|^2\right] \\
 &\leq \mathbb{E}_{i_kj_k} \left[f(x^k)
 + \left\langle \nabla f(x^k), \mathbb{E}_{l_k}[d_{i_kj_k}^{l_k}] \right\rangle + \frac{L}{2} \mathbb{E}_{l_k}[\|d_{i_kj_k}^{l_k}\|^2]\right] \\
& \leq \mathbb{E}_{i_kj_k} \Big[f(x^k)
 + \frac{\alpha_{k}}{2L} \left\langle \nabla_{y_{i_k}} f(x^k), \mathbb{E}_{l_k}[\nabla_{y_{j_k}}f_{l_k}(x^k) - \nabla_{y_{i_k}} f_{l_k}(x^k)] \right\rangle \\
&  \quad \quad +  \frac{\alpha_{k}}{2L} \left\langle \nabla_{y_{j_k}} f(x^k), \mathbb{E}_{l_k}[\nabla_{y_{i_k}}f_{l_k}(x^k) - \nabla_{y_{j_k}} f_{l_k}(x^k)] \right\rangle \\ 
 & \quad \quad -  \frac{\alpha_{k}}{L} \left\langle \nabla_{z_{i_k}} f(x^k), \mathbb{E}_{l_k}[\nabla_{z_{i_k}} f_{l_k}(x^k)] \right\rangle -  \frac{\alpha_{k}}{L} \left\langle \nabla_{z_{j_k}} f(x^k), \mathbb{E}_{l_k}[\nabla_{z_{j_k}} f_{l_k}(x^k)] \right\rangle + \frac{L}{2} \mathbb{E}_{l_k}[\|d_{i_kj_k}^{l_k}\|^2]\Big].
 \end{align*}
 Taking expectation over $l_k$, we get the following relationship:
 \begin{align*}
 \mathbb{E}[f(x^{k+1})|\eta_k] &\leq \mathbb{E}_{i_kj_k} \Big[f(x^k)
  +  \frac{\alpha_{k}}{2L} \left\langle \nabla_{y_{i_k}} f(x^k), \nabla_{y_{j_k}}f(x^k) - \nabla_{y_{i_k}} f(x^k) \right\rangle  \\
  & \quad \quad +  \frac{\alpha_{k}}{2L} \left\langle \nabla_{y_{j_k}} f(x^k), \nabla_{y_{i_k}}f(x^k) - \nabla_{y_{j_k}} f(x^k) \right\rangle \\ 
  & \quad \quad -  \frac{\alpha_{k}}{L} \left\langle \nabla_{z_{i_k}} f(x^k), \nabla_{z_{i_k}} f(x^k) \right\rangle -  \frac{\alpha_{k}}{L} \left\langle \nabla_{z_{j_k}} f(x^k), \nabla_{z_{j_k}} f(x^k) \right\rangle + \frac{L}{2} \mathbb{E}_{l_k}[\|d_{i_kj_k}^{l_k}\|^2]\Big].
\end{align*}

We first note that $\mathbb{E}_{l_k}[\|d_{i_kj_k}^{l_k}\|^2] \leq 8M^2\alpha_{k}^2/L^2$ since $\|\nabla f_l\| \leq M$. Substituting this in the above inequality and simplifying we get,
\begin{align}
\label{eq:stoch-conv-eq1}
\nonumber \mathbb{E}[f(x_{k+1})|\eta_k] &\leq f(x^k)
 - \alpha_{k} \nabla_y f(x^k)^\top ({\cal L} \otimes I_n) \nabla_y f(x^k)
  - \alpha_{k} \nabla_z f(x^k)^\top ({\cal D} \otimes I_n) \nabla_z f(x^k) + \frac{4M^2\alpha_{k}^2}{L} \\
& \leq f(x^k) - \alpha_{k} \nabla f(x^k)^\top \mathcal{K}  \nabla f(x^k) + \frac{4M^2\alpha_{k}^2}{L}.
\end{align}
Similar to Theorem~\ref{thm:conv}, we obtain a lower bound on $\nabla f(x^k)^\top \mathcal{K} \nabla f(x^k)$ in the following manner.
\begin{align*}
f(x^k) - f^* & \leq \langle \nabla f(x^k), x^k - x^* \rangle
\leq \|x^k - x^*\|_{\cal K}^*.\| \nabla f(x^k) \|_{\cal K} \\
& \leq R(x^0) \| \nabla f(x^k) \|_{\cal K}.
\end{align*}
Combining this with inequality Equation~\ref{eq:stoch-conv-eq1}, we obtain
\begin{align*}
\mathbb{E}[f(x_{k+1})|\eta_k] \leq f(x^k) - \alpha_{k} \frac{(f(x^k)-f^*)^2}{R^2(x^0)} + \frac{4M^2\alpha_{k}^2}{L}.
\end{align*}
Taking the expectation of both sides an denoting $\Delta_k = \mathbb{E}[f(x^k)] - f^*$ gives
\begin{align*}
\Delta_{k+1} \leq \Delta_k - \alpha_{k} \frac{\Delta_k^2}{R^2(x^0)} + \frac{4M^2\alpha_{k}^2}{L}.
\end{align*}
Adding these inequalities from $i=0$ to $i=k$ and use telescopy we get,
\begin{align*}
\Delta_{k+1} + \sum_{i=0}^{k} \alpha_i \frac{\Delta_{k}^2}{R^2(x^0)} \leq \Delta_0 + \frac{4M^2}{L} \sum_{i=0}^{k} \alpha_i^2.
\end{align*}
Using the definition of $\bar{x}_{k+1} = \arg\min_{0 \leq i \leq k+1} f(x_i)$, we get
\begin{align*}
\sum_{i=0}^{k} \alpha_i \frac{(\mathbb{E}[f(\bar{x}_{k+1}) - f^*])^2}{R^2(x^0)} \leq \Delta_{k+1} + \sum_{i=0}^{k} \alpha_i \frac{\Delta_{k}^2}{R^2(x^0)} \leq \Delta_0 + \frac{4M^2}{L} \sum_{i=0}^{k} \alpha_{i}^2.
\end{align*}
Therefore, from the above inequality we have,
\begin{align*}
\mathbb{E}[f(\bar{x}_{k+1}) - f^*] \leq R(x^0)\sqrt{\frac{(\Delta_0 + 4M^2\sum_{i=0}^{k} \alpha_i^2/L)}{\sum_{i=0}^{k} \alpha_i}}.
\end{align*}
Note that $\mathbb{E}[f(\bar{x}_{k+1}) - f^*] \rightarrow 0$ if we choose step sizes satisfying the condition that $\sum_{i=0}^\infty \alpha_i = \infty$ and $\sum_{i=0}^{\infty} \alpha_i^2 < \infty$. Substituting $\alpha_i = \sqrt{\Delta_0 L}/(2M\sqrt{i+1})$, we get the required result using the reasoning from \cite{Nemirovski09} (we refer the reader to Section 2.2 of \cite{Nemirovski09} for more details).
\end{proof}

\section{Proof of Theorem \ref{thm:asyn-conv}}

\begin{proof} For ease of exposition, we analyze the case where the unconstrained variables $z$ are absent. The analysis of case with $z$ variables can be carried out in a similar manner. Consider the update on edge $(i_k,j_k)$. Recall that $D(k)$ denotes the index of the iterate used in the $k^{\text{th}}$ iteration for calculating the gradients. Let $d^k = \frac{\alpha_k}{2L} \left(\nabla_{y_{j_k}} f(x^{D(k)}) - \nabla_{y_{i_k}} f(x^{D(k)}) \right)$ and $d_{i_kj_k}^k = x^{k+1} - x^{k} = U_{i_k}d^k - U_{j_k}d^k$. Note that $\|d_{i_kj_k}^k\|^2 = 2\|d^k\|^2$. Since $f$ is Lipschitz continuous gradient, we have
\begin{align*}
f(x^{k+1}) &\leq f(x^k) + \left\langle \nabla_{y_{i_k}y_{j_k}} f(x^k), d_{i_kj_k}^k \right\rangle + \frac{L}{2}\|d_{i_kj_k}^k\|^2 \\
&\leq f(x^k) + \left\langle \nabla_{y_{i_k}y_{j_k}} f(x^{D(k)}) + \nabla_{y_{i_k}y_{j_k}} f(x^{k}) - \nabla_{y_{i_k}y_{j_k}} f(x^{D(k)}), d_{i_kj_k}^k \right\rangle + \frac{L}{2}\|d_{i_kj_k}^k\|^2 \\
&\leq f(x^k) - \frac{L}{\alpha_k} \|d_{i_kj_k}^k\|^2 + \left\langle \nabla_{y_{i_k}y_{j_k}} f(x^{k}) - \nabla_{y_{i_k}y_{j_k}} f(x^{D(k)}), d_{i_kj_k}^k \right\rangle  + \frac{L}{2}\|d_{i_kj_k}^k\|^2 \\
&\leq f(x^k) - L\left(\frac{1}{\alpha_k} - \frac{1}{2}\right) \|d_{i_kj_k}^k\|^2 + \|\nabla_{y_{i_k}y_{j_k}} f(x^{k}) - \nabla_{y_{i_k}y_{j_k}} f(x^{D(k)})\| \|d_{i_kj_k}^k\| \\
&\leq f(x^k) - L\left(\frac{1}{\alpha_k} - \frac{1}{2}\right) \|d_{i_kj_k}^k\|^2 + L \|x^k - x^{D(k)}\| \|d_{i_kj_k}^k\|.
\end{align*}
The third and fourth steps in the above derivation follow from  definition of $d_{ij}^k$ and Cauchy-Schwarz inequality respectively. The last step follows from the fact the gradients are Lipschitz continuous. Using the assumption that staleness in the variables is bounded by $\tau$, i.e., $k - D(k) \leq \tau$ and definition of $d_{ij}^k$, we have
\begin{align*}
f(x^{k+1}) &\leq f(x^k) - L\left(\frac{1}{\alpha_k} - \frac{1}{2}\right) \|d_{i_kj_k}^k\|^2 + L\left(\sum_{t=1}^{\tau} \|d_{i_{k-t}j_{k-t}}^{k-t}\| \|d_{i_kj_k}^k\|\right) \\
&\leq f(x^k) - L\left(\frac{1}{\alpha_k} - \frac{1}{2}\right) \|d_{i_kj_k}^k\|^2 + \frac{L}{2} \left(\sum_{t=1}^{\tau} \left[ \|d_{i_{k-t}j_{k-t}}^{k-t}\|^2 + \|d_{i_kj_k}^k\|^2 \right]\right) \\
&\leq f(x^k) - L\left(\frac{1}{\alpha_k} - \frac{1+\tau}{2}\right) \|d_{i_kj_k}^k\|^2 + \frac{L}{2} \sum_{t=1}^{\tau} \|d_{i_{k-t}j_{k-t}}^{k-t}\|^2. \\
\end{align*}
The first step follows from triangle inequality. The second inequality follows from fact that $ab \leq (a^2 + b^2)/2$. Using expectation over the edges, we have
\begin{align} 
\label{eq:asyn-inter-bound}
\mathbb{E}[f(&x^{k+1})] \leq \mathbb{E}[f(x^k)] - L \left(\frac{1}{\alpha_k} - \frac{1+\tau}{2}\right) \mathbb{E}[\|d_{i_kj_k}^k\|^2] + \frac{L}{2} \mathbb{E}\left[\sum_{t=1}^{\tau} \|d_{i_{k-t}j_{k-t}}^{k-t}\|^2\right].
\end{align}
We now prove that, for all $k \geq 0$
\begin{align}
\label{eq:bounded-step}
\mathbb{E}\left[ \|d_{i_{k-1}j_{k-1}}^{k-1}\|^2\right] \leq \rho \mathbb{E}\left[ \|d_{i_{k}j_{k}}^k\|^2\right],
\end{align}
where we define $\mathbb{E}\left[ \|d_{i_{k-1}j_{k-1}}^{k-1}\|^2\right] = 0$ for $k = 0$.  
Let $w^t$ denote the vector of size $|E|$
such that $w_{ij}^t = \sqrt{p_{ij}}\|d_{ij}^t\|$
(with slight abuse of notation, we use $w^t_{ij}$ to denote the entry corresponding to edge $(i,j)$).
Note that $\mathbb{E}\left[ \|d_{i_{t}j_{t}}^t\|^2\right] = \mathbb{E}[\|w^t\|^2]$. We prove Equation~(\ref{eq:bounded-step}) by induction.

Let $u^k$ be a vector of size $|E|$ such that $u_{ij}^k = \sqrt{p_{ij}} \|d_{ij}^k - d_{ij}^{k-1}\|$. Consider the following:
\begin{align}
\label{eq:ind-hyp}
\nonumber \mathbb{E}[\|w^{k-1}\|]^2 - \mathbb{E}[\|w^{k}\|^2] &= \mathbb{E}[2\|w^{k-1}\|]^2 - \mathbb{E}[\|w^{k}\|^2 + \|w^{k-1}\|^2] \nonumber \\
&\leq 2\mathbb{E}[\|w^{k-1}\|^2] - 2\mathbb{E}[\langle w^{k-1}, w^{k} \rangle] \nonumber \\
\nonumber &\leq 2\mathbb{E}[\|w^{k-1}\| \|w^{k-1} - w^k\|] \\
\nonumber &\leq 2\mathbb{E}[\|w^{k-1}\|\|u^{k}\|] \leq 2\mathbb{E}[\|w^{k-1}\| \sqrt{2} \alpha_k \|x^{D(k)} - x^{D(k-1)}\|] \\
&\leq \sqrt{2} \alpha_k \sum_{t=\min(D(k-1),D(k))}^{\max(D(k-1),D(k))} \left(\mathbb{E}[\|w^{k-1}\|^2] + \mathbb{E}[\|d_{i_{t}j_t}^t\|^2]\right).
\end{align}

The fourth step follows from the bound below on $|u^k_{ij}|$
\begin{align*}
|u^k_{ij}| & = \sqrt{p_{ij}} \|d_{ij}^k - d_{ij}^{k-1}\| \\
& \leq \sqrt{p_{ij}} \|(U_i - U_j) \frac{\alpha_k}{2L} (\nabla_{y_i} f(x^{D(k)}) - \nabla_{y_j} f(x^{D(k)}) + \nabla_{y_j} f(x^{D(k-1)}) - \nabla_{y_i} f(x^{D(k-1)})) \| \\
& \leq \sqrt{2 p_{ij}} \alpha_k \|x^{D(k)} - x^{D(k-1)}\|.
\end{align*}

The fifth step follows from triangle inequality. We now prove \eqref{eq:bounded-step}: 
the induction hypothesis is trivially true for $k=0$. Assume it is true for some $k-1 \geq 0$. 
Now using Equation~(\ref{eq:ind-hyp}), we have
\begin{align*}
\mathbb{E}[\|w^{k-1}\|]^2 - \mathbb{E}[\|w^{k}\|^2] \leq \sqrt{2} \alpha_k (\tau+2) \mathbb{E}[\|w^{k-1}\|^2] + \sqrt{2} \alpha_k (\tau+2) \rho^{\tau+1}\mathbb{E}[\|w^{k}\|^2]
\end{align*}
for our choice of $\alpha_k$. The last step follows from the fact that $\mathbb{E}[\|d_{i_tj_t}^t\|^2] = \mathbb{E}[\|w^{t}\|^2]$ and mathematical induction. From the above, we get 
\begin{align*}
\mathbb{E}[\|w^{k-1}\|^2] \leq \frac{1 + \sqrt{2}{\alpha_k}(\tau+2)\rho^{(\tau+1)}}{1 - \sqrt{2}{\alpha_k}(\tau+2)} \mathbb{E}[\|w^{k}\|^2] \leq \rho \mathbb{E}[\|w^{k}\|^2].
\end{align*}
Thus, the statement holds for $k$. Therefore, the statement holds for all $k \in \mathbb{N}$ by mathematical induction. Substituting the above in Equation~(\ref{eq:asyn-inter-bound}), we get
\begin{align*}
\mathbb{E}[f(&x^{k+1})] \leq \mathbb{E}[f(x^k)] - L \left(\frac{1}{\alpha_k} - \frac{1+\tau + \tau\rho^{\tau}}{2}\right) \mathbb{E}[\|d_{i_kj_k}^k\|^2].
\end{align*}
This proves that the method is a descent method in expectation. Using the definition of $d_{ij}^k$, we have
\begin{align*}
\mathbb{E}[f(x^{k+1})] &\leq \mathbb{E}[f(x^k)] - \frac{\alpha_k^2}{4L}\left(\frac{1}{\alpha_k} - \frac{1+\tau+\tau \rho^{\tau}}{2}\right)\mathbb{E}[\|\nabla_{y_{i_k}} f(x^{D(k)}) - \nabla_{y_{j_k}} f(x^{D(k)})\|^2] \\
&\leq \mathbb{E}[f(x^k)] - \frac{\alpha_k^2}{4L}\left(\frac{1}{\alpha_k} - \frac{1+\tau+\tau \rho^{\tau}}{2}\right)\mathbb{E}[\|\nabla f(x^{D(k)}) - \nabla f(x^{D(k)})\|_{\cal K}^2] \\
&\leq \mathbb{E}[f(x^k)] - \frac{\alpha_k^2}{2R^2(x^0)}\left(\frac{1}{\alpha_k} - \frac{1+\tau+\tau \rho^{\tau}}{2}\right) \mathbb{E}[(f(x^{D(k)})-f^*)^2] \\
&\leq \mathbb{E}[f(x^k)] - \frac{\alpha_k^2}{2R^2(x^0)}\left(\frac{1}{\alpha_k} - \frac{1+\tau+\tau \rho^{\tau}}{2}\right) \mathbb{E}[(f(x^{k})-f^*)^2].
\end{align*}
The second and third steps are similar to the proof of Theorem \ref{thm:conv}.
The last step follows from the fact that the method is a descent method in expectation. Following similar analysis as Theorem \ref{thm:conv}, we get the required result.
\end{proof}

\section{Proof of Theorem \ref{thm:non-smooth-conv}}

\begin{proof}

Let $Ax = \sum_i x_i$.
Let $\tilde{x}_{k+1}$ be solution to the following optimization problem:
$$
\tilde{x}^{k+1} = \arg\min_{\{x| Ax = 0\}} \langle \nabla f(x^k), x - x^{k} \rangle + \frac{L}{2} \|x - x^{k} \|^2 + h(x).
$$
To prove our result, we first prove few intermediate results. We say vectors $d \in \mathbb{R}^n$ and $d' \in \mathbb{R}^n$ are conformal if $d_i d_i' \geq 0$ for all $i \in [b]$.
We use $d_{i_kj_k} = x^{k+1} - x^{k}$ and $d = \tilde{x}^{k+1} - x^{k}$. Our first claim is that for any $d$, we can always find conformal vectors whose sum is $d$ (see \cite{Necoara14}). More formally, we have the following result.
\begin{lemma}
For any $d \in \mathbb{R}^n$ with $Ad = 0$, we have a multi-set $S = \{d_{ij}'\}_{i \neq j}$ such that $d$ and $d_{ij}'$ are conformal for all $i \neq j$ and $i,j \in [b]$ i.e., $\sum_{i \neq j} d_{ij}' = d$, $Ad_{ij}' = 0$ and $d_{ij}'$ can be non-zero only in coordinates corresponding to $x_i$ and $x_j$.
\end{lemma}
\begin{proof}
We prove by an iterative construction, i.e., for every vector $d$ such that $Ad = 0$, we construct a set $S = \{s_{ij}\}$ ($s_{ij} \in \mathbb{R}^n$) with the required properties. We start with a vector $u^0 = d$ and multi-set $S^0 = \{s^0_{ij}\}$ and $s^0_{ij} = 0$ for all  $i \neq j$ and $i,j \in [n]$. At the $k^{\text{th}}$ step of the construction, we will have $Au^k = 0$, $As = 0$ for all $s \in S^k$, $d = u^k + \sum_{s \in S^k} s$ and each element of $s$ is conformal to $d$.

In $k^{\text{th}}$ iteration, pick the element with the smallest absolute value (say $v$) in $u^{k-1}$. Let us assume it corresponds to $y_p^j$. Now pick an element from $u^{k-1}$ corresponding to $y_q^j$ for $p \neq q \in [m]$ with at least absolute value $v$ albeit with opposite sign. Note that such an element should exist since $Au^{k-1} = 0$. Let $p_1$ and $p_2$ denote the indices of these elements in $u^{k-1}$. Let $S^k$ be same as $S^{k-1}$ except for $s^{k}_{pq}$ which is given by $s^k_{pq} = s^{k-1}_{pq} + r = s^{k-1}_{pq} + u_{p_1}^{k-1} e_{p_1} - u_{p_1}^{k-1} e_{p_2}$ where $e_i$ denotes a vector in $\mathbb{R}^n$ with zero in all components except in $i^{\text{th}}$ position (where it is one). Note that $Ar = 0$ and $r$ is conformal to $d$ since it has the same sign. Let $u^{k+1} = u^{k} - r$. Note that $Au^{k+1} = 0$ since $Au^k = 0$ and $Ar = 0$. Also observe that $As = 0$ for all $s \in S^{k+1}$ and $u^{k+1} = \sum_{s \in S^k} s = d$.

Finally, note that each iteration the number of non-zero elements of $u^{k}$ decrease by at least 1. Therefore, this algorithm terminates after a finite number of iterations. Moreover, at termination $u^{k} = 0$ otherwise the algorithm can always pick an element and continue with the process. This gives us the required conformal multi-set.
\end{proof}

Now consider a set $\{d_{ij}'\}$ which is conformal to $d$. We define $\hat{x}_{k+1}$ in the following manner:
\[ \hat{x}^{k+1}_i = \left\{ 
  \begin{array}{l l}
    x^k_i + d_{ij}' & \quad \text{if $(i,j) = (i_k,j_k)$}\\
    x^k_i & \quad \text{if $(i,j) \neq (i_k,j_k)$}
\end{array} \right.\]

\begin{lemma}
\label{lem:non-smooth-lemma}
For any $x \in \mathbb{R}^n$ and $k \geq 0$,
$$\mathbb{E}[\|\hat{x}^{k+1} - x^k\|^2 \leq \lambda(\|\tilde{x}^{k+1} - x^k\|^2 ).$$
We also have
$$
\mathbb{E}(h(\hat{x}^{k+1})) \leq (1 - \lambda) h(x^{k}) + \lambda h(\tilde{x}^{k+1}).
$$
\end{lemma}
\begin{proof}
We have the following bound:
\begin{align*}
\mathbb{E}_{i_kj_k}[\|\hat{x}^{k+1} - x^k\|^2  &= \lambda \sum_{i \neq j} \|d_{ij}'\|^2 
\leq \lambda \|\sum_{i \neq j} d_{ij}' \|^2
= \lambda \|d\|^2 = \lambda\|\tilde{x}^{k+1} - x^k\|^2.
\end{align*}
The above statement directly follows the fact that $\{d_{ij}'\}$ is conformal to $d$. The remaining part directly follows from \cite{Necoara14}.
\end{proof}

The remaining part essentially on similar lines as \citep{Necoara14}. We give the details here for completeness. From Lemma 1, we have
\begin{align*}
\mathbb{E}_{i_kj_k}[F(x^{k+1})] &\leq \mathbb{E}_{i_kj_k}[f(x^k) + \langle \nabla f(x^k), d_{i_kj_k} \rangle + \frac{L}{2} \|d_{i_kj_k}\|^2 + h(x^k + d_{i_kj_k})] \\
&\leq \mathbb{E}_{i_kj_k}[f(x^k) + \langle \nabla f(x^k), d_{i_kj_k}' \rangle + \frac{L}{2} \|d_{i_kj_k}'\|^2 + h(x^k + d_{i_kj_k}')] \\
&= f(x^k) + \lambda \left(\langle \nabla f(x), \sum_{i \neq j} d_{ij}' \rangle + \sum_{i \neq j} \frac{L}{2} \|d_{ij}'\|^2 + \sum_{i \neq j} h(x + d_{ij}') \right) \\
&\leq (1 - \lambda) F(x^k) + \lambda(f(x^k) + \langle \nabla f(x),d\rangle + \frac{L}{2} \|d\|^2 + h(x + d)) \\
&\leq \min_{\{y | Ay = 0\}} (1 - \lambda) F(x^k) + \lambda (F(y) + \frac{L}{2} \|y - x^{k}\|^2) \\
&\leq \min_{\beta \in [0,1]} (1 - \lambda) F(x^k) + \lambda ( F(\beta x^* + (1-\beta)x^{k}) + \frac{\beta^2 L}{2} \|x^k - x^*\|^2) \\
&\leq (1 - \lambda) F(x^k) + \lambda \left(F(x^k) - \frac{2(F(x^k) - F(x^*))^2}{LR^2(x^0)}\right).
\end{align*}

The second step follows from optimality of $d_{i_kj_k}$. The fourth step follows from Lemma~\ref{lem:non-smooth-lemma}. Now using the similar recurrence relation as in Theorem 2, we get the required result.
\end{proof}

\section{Reduction of General Case}
\label{sec:reduction}

In this section we show how to reduce a problem with linear constraints to the form of Problem 
\ref{eq:reduced-optimization-problem}
in the paper. For simplicity, we focus on smooth objective functions.
However, the formulation can be extended to composite objective functions
along similar lines.
Consider the optimization problem 
\begin{align*}
\min_x & \ f(x) \\
s.t. \  Ax = &\sum A_i x_i = 0,
\end{align*}
where $f_i$ is a convex function with an $L$-Lipschitz gradient.

Let $\nb{A_i}$ be a matrix with orthonormal columns satisfying $\mathrm{range}(\nb{A_i}) = \ker({A_i})$
, this can be obtained (e.g. using SVD).
For each $i$, define $y_i = A_i x_i$
and assume that the rank of $A_i$ is less than or equal to the dimensionality of $x_i$. 
\footnote{If the rank constraint is not satisfied
then one solution is to use a coarser partitioning of $x$
so that the dimensionality of $x_i$
is large enough.
}
Then we can rewrite $x$ as a function $h(y, z)$ satisfying
\begin{align*}
x_i = A_i^+ y_i + \nb{A_i} z_i,
\end{align*}
for some unknown $z_i$, where $C^+$ denote the pseudo-inverse of $C$. The problem then becomes
\begin{align}
\min_{y,z} g(y, z) \ \ s.t. \  \sum_{i=1}^N y_i = 0,
\label{eq:transformed}
\end{align}
where
\begin{align}
g(y, z) = f(\phi(y,z)) = f \left( \sum_i U_i (A_i^+ y_i + \nb{A_i} z_i) \right).
\label{eq:g}
\end{align}
It is clear that the sets $S_1 = \{x | Ax = 0\}$ and $S_2 = \{ \phi(y,z) | \sum_i y_i = 0 \}$
are equal and hence the problem defined in \ref{eq:transformed} is equivalent to
that in \ref{eq:opt}.

Note that such a transformation preserves convexity of the objective function. It is also easy to show that it preserves the block-wise Lipschitz continuity of the gradients as we prove
in the following result.

\begin{lemma}
Let $f$ be a function with $L_i$-Lipschitz gradient w.r.t $x_i$. 
Let $g(y, z)$ be the function defined in $\ref{eq:g}$. Then
$g$ satisfies the following condition
\begin{align*}
\|\nabla_{y_i} g(y,z) - \nabla_{y_i} g(y',z)\| \leq  \frac{L_i}{\sigma_{\min}^2(A_i)}\|y_i - y_i'\| \\
\|\nabla_{z_i} g(y,z) - \nabla_{z_i} g(y,z')\| \leq  L_i \|z_i - z_i'\|, \\
\end{align*}
where $\sigma_{\min}(B)$ 
denotes the minimum non-zero singular value of $B$.
\end{lemma}

\begin{proof}
We have
\begin{align*}
\|\nabla_{y_i} g(y,z) - \nabla_{y_i} g(y',z)\| & = \|(U_i A_i^+)^\top 
	[\nabla_x f(\phi(y,z)) - \nabla_x f(\phi(y', z))]\| \\
	& \leq \|A_i^+\| \| \nabla_i f(\phi(y,z)) - \nabla_i f(\phi(y',z)) \| \\
	& \leq L_i \|A_i^+\| \| A_i^+ (y_i - y_i') \| \leq L_i \|A_i^+\|^2 \|y_i - y_i'\| = \frac{L_i}{\sigma_{\min}^2(A_i)}\|y_i - y_i'\|,
\end{align*}
Similar proof holds for $\|\nabla_{z_i} g(y,z) - \nabla_{z_i} g(y,z')\|$
, noting that $\| \nb{A_i} \| = 1$.
\end{proof}

It is worth noting that this reduction is mainly used to simplify analysis.
In practice, however, we observed that 
an algorithm that operates directly on the original variables $x_i$ 
(i.e. Algorithm \ref{alg:composite-minimization})
converges much faster and is much less sensitive to the conditioning of $A_i$
compared to an algorithm that operates on $y_i$ and $z_i$. 
Indeed, with appropriate step sizes, Algorithm \ref{alg:composite-minimization}
minimizes, in each step, a tighter bound on the objective function compared to
the bound based \ref{eq:transformed} as stated in the following result. 

\begin{lemma}
Let $g$ and $\phi$ be as defined in \ref{eq:g}. And let 
\begin{align*}
d_i = A_i^+ d_{y_i} + \nb{A_i} d_{z_i}.
\end{align*}
Then, for any $d_i$ and $d_j$ satisfying $A_i d_i + A_j d_j = 0$ 
and any feasible $x = \phi(y, z)$ we have
\begin{align*}
& \langle \nabla_i f(x), d_i \rangle + \langle \nabla_j f(x), d_j \rangle
+ \frac{L_i}{2 \alpha} \| d_i \|^2 + \frac{L_j}{2 \alpha} \| d_j \|^2 \\
& \leq \langle \nabla_{y_i} g(y,z), d_{y_i} \rangle
+ \langle \nabla_{z_i} g(y,z), d_{z_i} \rangle
+ \langle \nabla_{y_j} g(y,z), d_{y_j} \rangle
+ \langle \nabla_{z_j} g(y,z), d_{z_j} \rangle \\
& + \frac{L_i}{2 \alpha \sigma^2_{\min} (A_i)} \| d_{y_i} \|^2 
+ \frac{L_i}{2 \alpha} \| d_{z_i} \|^2
+ \frac{L_j}{2 \alpha \sigma^2_{\min} (A_j)} \| d_{y_j} \|^2
+ \frac{L_j}{2 \alpha} \| d_{z_j} \|^2.
\end{align*}
\end{lemma}

\begin{proof}
The proof follows directly from the fact that
\begin{align*}
\nabla_{i} f(x) = {A_i^+}^\top \nabla_{y_i} g(y,z) + {\nb{A_i}}^\top \nabla_{z_i} g(y,z).
\end{align*}
\end{proof}